\documentclass[english,12pt,a4paper]{amsart}
\usepackage{amssymb}
\usepackage{amsfonts}
\usepackage{amscd}
\usepackage[all]{xypic}
\usepackage{color}
\usepackage{enumerate}
\usepackage{todonotes}
\usepackage{enumitem}

\usepackage[T1]{fontenc}
\RequirePackage[colorlinks=true, breaklinks=true, urlcolor= black, linkcolor= black, citecolor= black, bookmarksopen=true,linktocpage=true,plainpages=false,pdfpagelabels,unicode]{hyperref}

\usepackage{cleveref}

\CompileMatrices

\def\deg{{\rm deg}}

\def\ord{{\rm ord}}

\def\11{{\mathbf 1}}
\def\AA{{\mathbb A}}

\def\CC{{\mathbb C}}

\def\FF{{\mathbb F}}

\def\NN{{\mathbb N}}

\def\PP{{\mathbb P}}
\def\QQ{{\mathbb Q}}
\def\RR{{\mathbb R}}

\def\ZZ{{\mathbb Z}}

\def\cB{{\mathcal B}}

\def\cO{{\mathcal O}}

\def\cX{{\mathcal X}}
\def\cY{{\mathcal Y}}

\mathchardef\alphag="7C0B \mathchardef\betag="7C0C
\mathchardef\gammag="7C0D \mathchardef\deltag="7C0E
\mathchardef\varepsilong="7C22 \mathchardef\varphig="7C27
\mathchardef\psig="7C20 \mathchardef\zetag="7C10
\mathchardef\epsilong="7C0F \mathchardef\rhog="7C1A
\mathchardef\taug="7C1C \mathchardef\upsilong="7C1D
\mathchardef\iotag="7C13 \mathchardef\thetag="7C12
\mathchardef\pig="7C19 \mathchardef\sigmag="7C1B
\mathchardef\etag="7C11 \mathchardef\omegag="7C21
\mathchardef\kappag="7C14 \mathchardef\lambdag="7C15
\mathchardef\mug="7C16 \mathchardef\xig="7C18
\mathchardef\chig="7C1F \mathchardef\nug="7C17
\mathchardef\varthetag="7C23 \mathchardef\varpig="7C24
\mathchardef\varrhog="7C25 \mathchardef\varsigmag="7C26
\mathchardef\Omegag="7C0A \mathchardef\Thetag="7C02
\mathchardef\Sigmag="7C06 \mathchardef\Deltag="7C01
\mathchardef\Phig="7C08 \mathchardef\Gammag="7C00
\mathchardef\Psig="7C09 \mathchardef\Lambdag="7C03
\mathchardef\Xig="7C04 \mathchardef\Pig="7C05
\mathchardef\Upsilong="7C07

\newcounter{theoremcntr}[subsection]
\renewcommand*{\thetheoremcntr}{%
  \ifnum\value{subsection}=0 %
    \thesection
  \else
    \thesubsection
  \fi
  .\arabic{theoremcntr}%
}
\numberwithin{figure}{section}
\theoremstyle{plain}
\newtheorem{thm}{\protect\theoremname}[section]
\theoremstyle{plain} 
\newtheorem{maintheorem}[thm]{\protect\theoremname}
\theoremstyle{remark}
\newtheorem{remark}[thm]{\protect\remarkname}
\theoremstyle{definition}
\newtheorem{defn}[thm]{\protect\definitionname}
\theoremstyle{plain}

\theoremstyle{plain}
\newtheorem{cor}[thm]{\protect\corollaryname}
\theoremstyle{plain}
\newtheorem{prop}[thm]{\protect\propositionname}
\theoremstyle{plain}
\newtheorem{lem}[thm]{\protect\lemmaname}
\theoremstyle{plain}

\theoremstyle{plain}

\theoremstyle{plain}

\providecommand{\corollaryname}{Corollary}
\providecommand{\definitionname}{Definition}
\providecommand{\notationname}{Notation}
\providecommand{\lemmaname}{Lemma}
\providecommand{\propositionname}{Proposition}
\providecommand{\remarkname}{Remark}
\providecommand{\theoremname}{Theorem}
\providecommand{\conjecturename}{Conjecture}
\providecommand{\examplename}{Example}

\Crefname{prop}{Proposition}{Propositions}
\Crefname{thm}{Theorem}{Theorems}
\Crefname{lem}{Lemma}{Lemmas}
\Crefname{cor}{Corollary}{Corollaries}
\Crefname{maintheorem}{Theorem}{Theorems}
\def\boxit#1#2{\setbox1=\hbox{\kern#1{#2}\kern#1}%
\dimen1=\ht1 \advance\dimen1 by #1 \dimen2=\dp1 \advance\dimen2 by
#1
\setbox1=\hbox{\vrule height\dimen1 depth\dimen2\box1\vrule}%
\setbox1=\vbox{\hrule\box1\hrule}%
\advance\dimen1 by .4pt \ht1=\dimen1 \advance\dimen2 by .4pt
\dp1=\dimen2 \box1\relax}

\mathchardef\alphag="7C0B \mathchardef\betag="7C0C
\mathchardef\gammag="7C0D \mathchardef\deltag="7C0E
\mathchardef\varepsilong="7C22 \mathchardef\varphig="7C27
\mathchardef\psig="7C20 \mathchardef\zetag="7C10
\mathchardef\epsilong="7C0F \mathchardef\rhog="7C1A
\mathchardef\taug="7C1C \mathchardef\upsilong="7C1D
\mathchardef\iotag="7C13 \mathchardef\thetag="7C12
\mathchardef\pig="7C19 \mathchardef\sigmag="7C1B
\mathchardef\etag="7C11 \mathchardef\omegag="7C21
\mathchardef\kappag="7C14 \mathchardef\lambdag="7C15
\mathchardef\mug="7C16 \mathchardef\xig="7C18
\mathchardef\chig="7C1F \mathchardef\nug="7C17
\mathchardef\varthetag="7C23 \mathchardef\varpig="7C24
\mathchardef\varrhog="7C25 \mathchardef\varsigmag="7C26
\mathchardef\Omegag="7C0A \mathchardef\Thetag="7C02
\mathchardef\Sigmag="7C06 \mathchardef\Deltag="7C01
\mathchardef\Phig="7C08 \mathchardef\Gammag="7C00
\mathchardef\Psig="7C09 \mathchardef\Lambdag="7C03
\mathchardef\Xig="7C04 \mathchardef\Pig="7C05
\mathchardef\Upsilong="7C07

\def\e{\varepsilon}

\DeclareMathOperator*{\lcm}{lcm}

\DeclareMathOperator{\charac}{char}
\DeclareMathOperator{\dire}{dir}

\def\ord{{\rm ord}}

\definecolor{orange}{rgb}{1,0.5,0}

\definecolor{blue}{rgb}{0,0,1}

\thanks{The authors would like to thank Tim Browning, Pierre Dèbes, Jan Denef, Roger Heath-Brown, François Loeser, Kien Nguyen, Lillian Pierce, Jonathan Pila, Damaris Schindler, Julien Sebag, and Jean-Pierre Serre for encouragements and stimulating discussions on the topics of the paper. R.~C.~was partially supported by KU Leuven Internal Funds C16/23/010 and acknowledges the support of the CDP C2EMPI, as well as the French State under the France-2030 programme, the University of Lille, the Initiative of Excellence of the University of Lille, the European Metropolis of Lille for their funding and support of the R-CDP-24-004-C2EMPI project. T.~B.~was supported by FWO Flanders (Belgium) with grant number 1131925N. T.S.~was 
supported by the Herchel-Smith Fund. F.~V.~was 
supported by the Humboldt foundation.}

\title[Serre's question on projective thin sets]{Serre's question on thin sets in projective space}

\author[Buggenhout]{Tijs Buggenhout}
\address{KU Leuven, Department of Mathematics,
B-3001 Leu\-ven, Bel\-gium}
\email{Tijs.Buggenhout@kuleuven.be}

\author[Cluckers]{Raf Cluckers}
\address{Univ.~Lille, CNRS, UMR 8524 - Laboratoire Paul Painlev\'e, F-59000 Lille, France, and
KU Leuven, Department of Mathematics, B-3001 Leu\-ven, Bel\-gium}
\email{Raf.Cluckers@univ-lille.fr}
\urladdr{http://rcluckers.perso.math.cnrs.fr/}

\author[Salberger]{Per Salberger}
\address{Mathematical Sciences, Chalmers University of Technology, SE-412 96 Göteborg, Sweden}
\email{salberg@chalmers.se}

\author[Santens]{Tim Santens}
\address{University of Cambridge, Department of Pure Mathematics and Mathematical Statistics, Cambridge, United Kingdom}
\email{ts996@cam.ac.uk}
\urladdr{https://sites.google.com/view/timsantens/}

\author[Vermeulen]{Floris Vermeulen}
\address{University of Münster, Mathematics Münster, Einsteinstrasse 62, 48149 Münster, Germany}
\email{florisvermeulen.math@gmail.com}
\urladdr{https://sites.google.com/view/floris-vermeulen}


\begin{document}

\begin{abstract}
We answer a question of Serre from the 1980s on rational points of bounded height on projective thin sets, in degree at least $4$. For degrees $2$ and $3$ we improve the known bounds in general. The focus is on thin sets of type II, namely corresponding to the images of ramified dominant quasi-finite covers of projective space, as thin sets of
type I are already well understood via dimension growth results by the third author in 2002 (published in 2023) by a global variant of Heath-Brown's $p$-adic determinant method. For type II, we obtain a uniform affine variant of Serre's question which implies the projective case and for which the implicit constant is furthermore polynomial in the degree. We are able to avoid logarithmic factors when the degree is at least $5$ and we prove our results over any global field, of any characteristic. A key ingredient for obtaining the affine variant comes from Binyamini-Cluckers-Novikov (2024) and Binyamini-Cluckers-Kato (2025) where a question of the third author was answered by providing bounds, for 
rational points 
on irreducible curves, which are quadratic in the degree. A second key ingredient is an adaptation of Salberger's global determinant method to the case of weighted polynomials. The third key ingredient is the design of our affine variant of Serre's question, for weighted polynomials which are not necessarily weighted homogeneous.
\end{abstract}

\maketitle
	
\section{Introduction}

In his 1980 - 1981 course at the Collège de France, Serre defined thin subsets, `ensembles minces', of affine and projective space over a number field. He furthermore studied the number of rational points with bounded height lying in thin sets. He distinguishes two types of thin sets, called of type I and type II, and shows that every thin set is included in a finite union of thin sets of type I and type II, see \cite{Serre-Mordell}. Roughly speaking, a thin set of type I consists of rational points in a proper Zariski-closed subset, and a thin set of type II consists of the images of rational points under a ramified dominant quasi-finite cover. 
Serre briefly mentions that type II is the most interesting case of a thin set on page 121, Section 9 of loc.~cit., as they are linked to irreducibility and the specialisation of Galois groups. He shows, following Cohen \cite{Cohen-S-D}, that one can win one half in the exponent of the trivial upper bound for the number of rational points of bounded height on projective thin sets of type I and type II. He raises the question on page 178 of whether these bounds 
can be improved by winning a further half on Cohen's exponent, for type II (below Theorem 3) as well as for type I (below Theorem 4). Let us sketch the state of the art for type I and type II thin sets prior to our work, where we refer to \cite{BPW-survey} for a more complete recent survey on the topic of thin sets.

\par

For projective thin sets of type I, Serre's question is now a theorem by the third author \cite{Salberger-dgc} coined dimension growth and was  announced in 2002 \cite{Salberger-dgc2009}. In fact, \cite{Salberger-dgc} shows more: the author almost always gets uniform upper bounds and thus he almost completely settles a uniform variant of Serre's question raised by Heath-Brown \cite{Heath-Brown-cubic}; only the case of degree $3$ remains open for the uniform bounds for type I thin sets. Salberger \cite{Salberger-dgc} completed the cases of degrees $3$ (non-uniformly), $4$, and $5$ after the cases of degree $2$, resp.~greater than $5$ were already obtained uniformly in \cite{Heath-Brown-Ann}, resp.~\cite{Brow-Heath-Salb}. To this end, Salberger \cite{Salberger-dgc} developed a powerful global variant of the $p$-adic determinant method from \cite{Heath-Brown-Ann}. This global determinant method from \cite{Salberger-dgc} has recently worked in tandem with a solution in \cite{BCN-d}, \cite{BinCluKat} of a question raised by the third author \cite[below Theorem 0.12]{Salberger-dgc2009,Salberger-dgc} on quadratic dependence on the degree. This tandem has resulted in refinements and generalizations of dimension growth results, see e.g.~
\cite{CDHNV-dgc} \cite{Verz-smooth}; in particular, it achieves for degree at least $4$ the aspect of Serre's question that no $\e$ should appear in the exponent of the bounds. It is precisely this tandem that we put to work for thin sets of type II, together with the design of an affine variant of Serre's question. 

For projective thin sets of type II however, relatively little progress was made on Serre's question since his course at the Collège de France, and not so much evidence was known before our work (see further below for an account on how our joint work emerged). The evidence mainly consisted of low dimensional cases of Serre's question about type II projective thin sets, namely in $\PP^1$ (see Section 9.7 of \cite{Serre-Mordell}), and in $\PP^2$ for degree at least $3$ (see Theorem 2 of \cite{Broberg03}). Under extra conditions, high dimensional cases are obtained in \cite{H-BPier} and eventual results in \cite{BonPier} \cite{BonPiWoo}, with noise in the exponent going to zero when the dimension grows. In \cite{Ghosh2023Bounds}, a special but important case of degree $2$ in $\PP^2$ (related to del Pezzo surfaces in $\PP^3$) is treated with just $1/10$ as noise in the exponent.

We solve Serre's question for rational points of bounded height on projective thin sets of type II of degree at least $4$, in all dimensions and over all global fields (see \Cref{thm:main-Serre,thm:main.global}). For degrees $2$ and $3$ we improve upon the known bounds 
in the general case (some special cases are known in stronger forms).

Let us first describe the key proof ingredients used in this paper, and then give an account on the history of our joint work.
A key ingredient is the recent solution in \cite{BCN-d}, \cite{BinCluKat} of a question of \cite{{Salberger-dgc2009,Salberger-dgc}} with bounds for rational points of bounded height on curves, which are quadratic in the degree (see \Cref{thm:BCK} below).  Another key ingredient is the adaptation of the global determinant method from \cite{Salberger-dgc} to the case of weighted polynomials (see \Cref{thm:aux.general,thm:aux.affine.general}). These two ingredients allow us to show a newly designed affine variant of Serre's question (see \Cref{thm:mainFYX,thm:main.unif.global}) for weighted polynomials which are not necessarily weighted homogeneous. Note that a well-designed affine variant also plays a key role for the study of type I thin sets, since the work \cite{Brow-Heath-Salb}. By combining the quadratic dependence from \cite{BCN-d} \cite{BinCluKat}, our weighted variant of the global determinant method from \cite{Salberger-dgc}, and our affine variant of Serre's question, we can avoid $\e$ completely, and we can even avoid logarithmic factors from degree $5$ on, all in line with results for type I from 
\cite{CDHNV-dgc}. We will also use a classical reduction to a specific kind of covers detailed by Serre in \cite[Section 9.1]{Serre-Mordell} in the affine case, and adapted to the projective case by Broberg, see \Cref{lem:Broberg} below. 

Through email communication in early 2026 with Salberger, after the first arXiv version \cite{BCSV-1} and after noticing the mention in \cite[below Remark 1]{Ghosh2023Bounds}, the other authors learnt from the third author that he already had a proof with a similar adaption of his global determinant method from \cite{Salberger-dgc2009,Salberger-dgc} to treat Serre's question over $\QQ$ for thin projective sets of type II in degree at least $4$. He presented this proof in his talk at a 2010 ETH Zürich workshop 
\cite{Salberger-2010-thin-ETH}. 
These are the main differences between the approach from \cite{Salberger-2010-thin-ETH} and the one of the present paper: in this paper we avoid logarithms from degree $5$ on and completely avoid $\varepsilon$-powers; we obtain polynomial dependence in the degree for an affine variant; we work over general global fields of any characteristic; finally, we streamline the approach for all degrees by relying on the mentioned quadratic results, where Salberger needed a more involved approach for degrees $2$, $3$, and $4$ (this more involved approach is also visible in \cite{Salberger-dgc} for degrees $3$ and $4$). We decided it would do justice to the subject and its history to join forces and become coauthors for this work.

We will next explain our results and Serre's question for type II thin sets into detail, over $\QQ$ (for general global fields, see \Cref{sec:global}).

\subsection{}
For any morphism $f : X\to \PP^n$ (defined over $\QQ$) from a scheme $X$ to $n$-dimensional projective space and any $B\ge 1$ one can define a counting
function
$$
N_{\rm cover}(f,B) := \# \{  x\in f( X(\QQ)) \mid H(x) \leq B\},
$$
where $H(x)$ equals the standard multiplicative height on $\PP^n(\QQ)$, namely, $H(x)$ is the maximum of the $|a_i|$ over $i=0,\ldots,n$ when we have homogeneous coordinates $x=(a_0:\ldots :a_n)$ with a tuple of integers $a_i$ which is coprime.

Let us furthermore assume that $X$ is integral and that $f$ is quasi-finite, dominant and of degree at least $2$. Under these assumptions, Serre \cite{Serre-Mordell} calls $f( X(\QQ))$ a thin set of type II and raises the question of whether one has for all $B\ge 2$
\begin{equation}\label{eq:Serre's;question}
N_{\rm cover}(f,B) \le c B^{n}(\log B)^\kappa
\end{equation}
for some constants $c$ and $\kappa$ depending on $f$. This is Serre's question on thin sets of type II from page 178 of \cite{Serre-Mordell}, which are in a way the most interesting case of thin sets 
as they play an important role for Hilbertian fields \cite{Lang-dioph,Serre-Mordell,Fried-Jarden-4}.
In this paper we show that the upper bounds from (\ref{eq:Serre's;question}) hold for degree $d\geq 4$ and any $n$, and we also prove the corresponding result over any global field $K$ (of any characteristic), see Theorems \ref{thm:main-Serre} and \ref{thm:main.global}, where for $d\geq 5$ we obtain (\ref{eq:Serre's;question}) without logarithmic factors.

For bounds as (\ref{eq:Serre's;question}) it is always a good first step to compare to trivial upper bounds, which in this situation are of the form
\begin{equation}\label{eq:triv}
N_{\rm cover}(f,B) \le  cB^{n+1}
\end{equation}
for some constant $c$ depending only on $n$. Using the large sieve method, Cohen \cite{Cohen-S-D} and Serre \cite[p.~178]{Serre-Mordell} find bounds with 1/2 subtracted from the exponent of the trivial upper bound from (\ref{eq:triv}), up to allowing some extra logarithmic factors; similar bounds were also achieved for affine thin sets, which are optimal in the affine case in general, see \cite[p.~177]{Serre-Mordell}. The question was then raised whether one can even win $1$ in the exponent of the trivial upper bound in the projective case, hence the form of~(\ref{eq:Serre's;question}).
We address this question as follows.

\begin{maintheorem}[Serre's question on thin sets of type II]\label{thm:main-Serre}
Let $f:X\to\PP^n$ be a morphism defined over $\QQ$, quasi-finite, dominant and of degree $d$ at least $2$ and assume that $X$ is integral.
Then one has for all $B\ge 2$
\begin{equation}\label{eq:Serre's;question.d>4}
N_{\rm cover}(f,B) \ll_{f} \begin{cases}
B^{n} & \text{ if } d\geq 5, \\
B^{n} \log(B)^{O(1)} & \text{ if } d = 4, \\
B^{n-1+2/\sqrt{d}}  \log(B)^{O(1)} & \text{ if } d=2,3.
\end{cases}
\end{equation}
\end{maintheorem}

In work of Browning, Heath-Brown and Salberger \cite{Brow-Heath-Salb} on thin sets of type I in projective space (concerning the dimension growth conjecture), an important insight was that there is an underlying \emph{affine} situation that implies the projective case. That reduction
to an affine situation inspired us to look for an affine situation where one can also win $1$ in the exponent of the corresponding trivial upper bound.
This needs a careful design since in general one can only win $1/2$ in the affine case, see \cite[p.~177]{Serre-Mordell}.  Affine situations are advantageous since they are more easily amenable to induction on the dimension, as in \cite{Brow-Heath-Salb}. To design this affine case, we were inspired by the classical, basic reduction to a specific kind of covers from \cite[Section 9.1]{Serre-Mordell} in the affine case, and from Broberg \cite[Lemma~3]{Broberg03} in the projective case, recalled as \Cref{lem:Broberg} below.

\subsection{The key affine situation leading to Serre's question}

We explain our affine result that underlies Serre's question about thin sets in projective space.  Fix integers $n\ge 1$, $d\ge 2$, and $e \ge 1$ and consider a polynomial
\begin{equation}\label{eq:FYX}
F_{\rm top}(Y, X) = Y^{d} + Y^{d-1}f_1(X) + \ldots + Y f_{d-1}(X) + f_d(X)
\end{equation}
in which for each $i$ with $1 \le i \le d$, $f_i$ lies in $\ZZ[X] = \ZZ[X_1, \ldots , X_n]$ and is homogeneous of degree $e \cdot i$ (or, $f_i$ is zero). 
Note that $F_{\rm top}(Y^e,X)$ is homogeneous of degree $de$, or equivalently $F_{\rm top}$ is weighted homogeneous of degree $de$ with weights $(e, 1, \ldots, 1)$. Consider a polynomial $F_0(Y,X)$ over $\ZZ$ such that $F_0(Y^e,X)$ is of degree strictly less than $de$ and put
\begin{equation}\label{eq:GYX}
F(Y,X) := F_{\rm top}(Y,X) + F_0(Y,X) .
\end{equation}
Under the assumption that $F_{\rm top}$ is absolutely irreducible, we will bound
$$
N_{\rm cover}^{\rm aff}(F, B) := \# \{x \in  [-B, B]^n \cap \ZZ^n : \exists y \in \ZZ \mbox{ such that } F(y, x) = 0\}.
$$
By the classical reduction recalled as \Cref{lem:Broberg} below, 
Serre's question on thin sets of type II essentially follows from the following theorem, up to some small extra term in the exponent when the degree $d$ is $2$ or $3$.
Write $||F||$ for the maximum of the absolute values of the coefficients of $F$.

\begin{maintheorem}[Affine result for thin sets of type II]\label{thm:mainFYX}
Let $F_{\rm top}$ and $F$ be as above in (\ref{eq:FYX}), (\ref{eq:GYX}) for some integers $n\ge 2$, $d\ge 2$ and $e \ge 1$. Suppose that $F_{\rm top}$ is absolutely irreducible.
Then, for any $B\ge 2$ one has
\begin{equation}\label{eq:NFB-Serre}
N_{\rm cover}^{\rm aff}(F, B)  \ll_{n,d,e,||F||} \begin{cases}
B^{n-1} & \text{ if } d\geq 5, \\
B^{n-1} \log(B)^{O(1)} & \text{ if } d = 4, \\
B^{n-2+2/\sqrt{d}}  \log(B)^{O(1)} & \text{ if } d=2,3.
\end{cases}
\end{equation}
\end{maintheorem}

Bounding $N_{\rm cover}^{\rm aff}(F, B)$ as in \Cref{thm:mainFYX} leads us to bounding the value 
$$
N_{\rm aff}(F; B^e,B,\ldots,B)
$$
which is defined as
$$
\#\{(y,x)\in \ZZ^{1+n} \mid y\in [-B^e, B^e],\ x \in  [-B, B]^n ,\   F(y, x) = 0\}
$$
and which we can bound uniformly in $F$ as follows.
\begin{maintheorem}[Uniform Affine result]\label{thm:mainFYX:unif}
Let $F_{\rm top}$ and $F$ be as above in (\ref{eq:FYX}), (\ref{eq:GYX}) for some integers $n\ge 2$, $d\ge 2$ and $e \ge 1$. Suppose that $F_{\rm top}$ is absolutely irreducible.
Then, for any $B\ge 2$ one has
\begin{equation}\label{eq:NFB-Serre:unif}
N_{\rm aff}(F; B^e,B,\ldots,B) \ll_{n,e} \begin{cases}
d^{e+O(n)} B^{n-1} & \text{ if } d\geq 5, \\
B^{n-1} \log(B)^{O(1)} & \text{ if } d = 4, \\
B^{n-2+2/\sqrt{d}}  \log(B)^{O(1)} & \text{ if } d=2,3.
\end{cases}
\end{equation}
\end{maintheorem}

Note that in the above theorems as well as below, $O(1)$ stands for an absolute constant, $O(n)$ stands for a constant multiple of $n$, and $O_n(1)$ stands for a constant depending on $n$ only, as is common notation.

\subsection{}

Note that Serre's question on thin sets of type I (about rational points of bounded height on irreducible projective varieties of degree at least $2$) is fully settled over $\QQ$ (see \Cref{thm:BCK}~\ref{item:3} below for a recent variant from \cite{BinCluKat} which is furthermore quadratic in the degree). Indeed, Salberger \cite[Thm.~0.1]{Salberger-dgc} finalized the cases of degrees $3$, $4$ and $5$ that were left open in \cite{Brow-Heath-Salb} \cite{Heath-Brown-Ann}, and the factor $B^\varepsilon$ was removed for degree at least $4$ in subsequent refinements. These results  are often called dimension growth results, and they have meanwhile been refined and generalized, see \cite{CCDN-dgc} \cite{Vermeulen:p} \cite{Pared-Sas} \cite{Verm:aff}
\cite{BinCluKat} \cite{CDHNV-dgc}. In particular, they now hold  for all global fields $K$ of any characteristic (with some small noise in the exponent in degree three) and have polynomial dependence on the degree.

\subsection{}\label{subsec:special}

Let us describe some special cases of our theorems which are sometimes known in stronger forms.
\Cref{thm:mainFYX:unif} with  $e=1$  is already known and follows from affine variants of dimension growth results.
Indeed, \Cref{thm:mainFYX:unif} with $e=1$ follows from \cite[Thm.~4]{CCDN-dgc} when $d\ge 5$ and from \cite[Thm.~5]{BinCluKat} (cf.~\Cref{thm:BCK}~\ref{item:2} below) when $d=3$, $4$. When $e=1$ and $d\le 3$ in \Cref{thm:mainFYX:unif}, one has even better bounds than (\ref{eq:NFB-Serre:unif}), see \cite{Heath-Brown-Ann}, \cite{Comtat}, \cite{Salb:d3} and \cite{Salberger-dgc}. 

\par
In degrees $2$ and $3$, if one assumes some extra smoothness or allowability conditions on $F$ (and $n\ge 6$ for degree $3$), then \cite{BonPier} \cite{BonPiWoo} achieve better bounds than those from \Cref{thm:mainFYX,thm:mainFYX:unif}. Likewise, in degrees $2$ and $3$, if one assumes some form of smoothness and cyclicity and that $n$ is large enough, then  \cite{H-BPier} achieves better bounds than those from \Cref{thm:main-Serre}.

\par When $n=1$, \Cref{thm:main-Serre} is known in a sharper form, see Section 9.7 of \cite{Serre-Mordell} whose treatment is adapted below in \Cref{thm:PP1} to any global field. Also the case with $n=2$ and $d\le 3$ of \Cref{thm:main-Serre} is known in a sharper form, see Theorem 2 of \cite{Broberg03}.

\par
The case $d = 2$ of Serre's question is known to be hard in general. For example, a double cover $X \to \PP^2$ which is smooth and ramified in a quartic curve is a del Pezzo surface of degree $2$ and optimal upper bounds are still an open problem, where $O(B^{2 + \frac{1}{10} + \varepsilon})$ is achieved in \cite{Ghosh2023Bounds}.
In general, our bounds of the above theorems are best known up to date in any degree $d\ge 2$. 

\subsection{}

We now sketch some proof ideas and the structure of the paper.
We recall some results from the literature in \Cref{sec:prelims}. We then develop, in \Cref{sec:auxi;pol},  the determinant method for weighted polynomials (not necessarily weighted homogeneous), both in a weighted projective and affine situation, by adapting a whole history of ideas and techniques, mostly from \cite{Heath-Brown-Ann} \cite{Ellenb-Venkatesh} \cite{Salberger-dgc} \cite{Walsh} \cite{CCDN-dgc}.  In \Cref{sec:twisted} we count separately on a specific kind of curves, called twisted lines; this can be seen as a weighted variant of Proposition 1 of \cite{Brow-Heath-Salb} with $D=1$.  In \Cref{sec:main-proofs} we develop the proof of \Cref{thm:mainFYX:unif}, which has the affine surface case as induction basis. To show the induction basis, the affine surface is cut by an auxiliary equation coming from \Cref{sec:auxi;pol}, and on the curves in the intersection we count following \Cref{sec:twisted} for the twisted lines, we do a counting for weighted  curves (from \Cref{sec:auxi;pol}) for the curves of slightly higher degree, and quadratic counting (in the degree of the curve) from \cite{BCN-d} \cite{BinCluKat}, cf.~\ref{item:1} of \Cref{thm:BCK}, for the higher degree curves. The general dimensional case of \Cref{thm:mainFYX:unif} then follows from intersecting with well-chosen hyperplanes and induction on the dimension. \Cref{thm:mainFYX,thm:main-Serre} are derived from \Cref{thm:mainFYX:unif}. 

\par

In \Cref{sec:global}, we generalize our main results to all global fields $K$ (of any characteristic) instead of just $K=\QQ$, based on the set-up of \cite{Pared-Sas}.

\section{Preliminaries}\label{sec:prelims}

We recall various results that we will need from the literature.

For $X\subset \PP^n$ a projective variety we denote by
\[
N(X,B) = \# \{x\in X(\QQ)\mid H(x)\leq B\}
\]
the usual counting function.
Similarly, if $X\subset \AA^n$ is an affine variety then we define
\[
N_{\rm aff}(X,B) = \# (X\cap \ZZ^n\cap [-B,B]^n) = N_{\rm aff}(X; B, \ldots, B).
\]
If $X$ is defined by a polynomial $f$ then we will also write $N(f,B)$ or $N_{\rm aff}(f,B)$ for $N(X,B)$ respectively $N_{\rm aff}(X,B)$.

First recall the Schwarz--Zippel bound, which is a trivial bound for counting integral or rational points of bounded height.

\begin{prop}\label[prop]{prop:schwarz-zippel}
\begin{enumerate}[label=(\roman*)]
\item Let $X\subset \AA^n_\QQ$ be an affine variety of degree $d$ and pure dimension $m$.
Then
\[
N_{\rm aff}(X,B) \leq d(2B+1)^m.
\]
\item Let $X\subset \PP^n_\QQ$ be a projective variety of degree $d$ and pure dimension $m$.
Then
\[
N(X,B)\leq d(2B+1)^{m+1}.
\]
\end{enumerate}
\end{prop}
\begin{proof}
See e.g.~\cite[Thm.\,1]{BrowningHeathBrown-Crelle}.
\end{proof}

We recall the following bounds which are 
optimal in the degree for counting integral and rational points on curves and hypersurfaces.

\begin{thm}[{{Binyamini--Cluckers--Kato~\cite{BinCluKat}}}]\label{thm:BCK}
\begin{enumerate}[label=(\roman*)]
\item\label{item:1} Let $C\subset \AA^n$ be an irreducible curve of degree $d\ge 1$. Then
\[
N_{\rm aff}(C,B)\ll d^2 B^{1/d}(\log B)^{O(1)}.
\]
\item\label{item:2} Let $n\geq 3$, let $f\in \ZZ[x_1, \ldots, x_n]$ be of degree $d\geq 4$ and assume that the degree $d$ part $f_d$ of $f$ is absolutely irreducible.
Then
\[
N_{\rm aff}(f,B)\ll d^2 B^{n-2} (\log B)^{O_n(1)}.
\]
\item\label{item:3} Let $X\subset \PP^n$ be an irreducible hypersurface of degree $d\geq 4$.
Then
\[
N(X,B)\ll d^2 B^{n-1} (\log B)^{O_n(1)}.
\]
\end{enumerate}
\end{thm}

Note that item \ref{item:1} of \Cref{thm:BCK} already appears in \cite{BCN-d} and sharpens results going back to \cite{Bombieri-Pila}. Also note that several variants of \Cref{thm:BCK} are known, with generalizations or without powers of $\log B$ when $d>4$ (but then with a higher exponent of $d$), see e.g.~\cite{CCDN-dgc} \cite{Verm:aff} \cite{CDHNV-dgc}. The optimality of the factor $d^2$ in the upper bounds of items \ref{item:1} and \ref{item:2} is shown in Section 6 of \cite{CCDN-dgc}, and the eventual optimality of $d^2$  in item \ref{item:3} (namely when $n$ grows) is obtained in \cite{CGlaz}, with an open question about possible improvements for fixed $n$.

We recall the Bombieri--Vaaler theorem.

\begin{thm}[{{Bombieri--Vaaler~\cite{Bombi-Vaal}}}]\label{thm:BV}
Let $A\in \ZZ^{s\times r}$ have full rank, for $r>s$.
Then there exists a non-trivial integer solution $(x_1, \ldots, x_r)$ to the system $AX = 0$ satisfying
\[
\max_i |x_i|\leq \left(D^{-1} \sqrt{|\det(AA^T)|}\right)^{1/(r-s)},
\]
where $D$ is the greatest common divisor of all $s\times s$ minors of $A$.
\end{thm}

The following is Broberg's projective variant of Serre's classical, affine reduction from \cite[end of Section 9.1]{Serre-Mordell}. 
It allows one to reduce Serre's question on projective thin sets of type II to a special kind of cover, and then further down to counting integral solutions to weighted polynomials in skew boxes as we do in \Cref{thm:mainFYX:unif}.

\begin{lem}[{{Broberg~\cite[Lem.\,3]{Broberg03}}}]\label[lem]{lem:Broberg}
Let $f: X\to \PP^n$ be a quasi-finite dominant morphism of degree $d$ defined over $\QQ$ and assume that $X$ is integral.
Then there exists a proper closed subset $V\subset X$, a positive integer $e$ and an irreducible polynomial $F\in \ZZ[Y, X_0, \ldots, X_n]$ of the form
\[
F = Y^d + f_1(X)Y^{d-1} + \ldots + f_d(X),
\]
where $f_i$ is homogeneous of degree $ei$, such that
\[
N_{\rm cover}(f,B)\ll N(f(V), B) + N_{\rm cover}^{\rm aff}(F,B).
\]
\end{lem}

We may assume that $F$ is absolutely irreducible, in view of the following.

\begin{lem}\label[lem]{lem:irre.abs.irre}
Let $e,d$ be positive integers and let $F\in \ZZ[Y, X_0, \ldots, X_n]$ be of the form
\[
F = Y^d + f_1(X)Y^{d-1} + \ldots + f_d(X),
\]
where $f_i$ is homogeneous of degree $ei$.
If $F$ is irreducible but not absolutely irreducible, then
\[
N_{\rm cover}^{\rm aff}(F,B)\ll_{n,d,e} B^n.
\]
\end{lem}

\begin{proof}
This follows from an argument identical to~\cite[Cor.\,1]{Heath-Brown-Ann}, see also~\cite[Sec.\,4]{Walsh}.
\end{proof}

\section{Auxiliary polynomials}\label{sec:auxi;pol}

In this section we develop the determinant method for weighted polynomials.
This type of determinant method also appears in work of Xiao~\cite{Xiao}, but our treatment is somewhat different.
This will be used to construct auxiliary polynomials catching all integral or rational points of bounded height.

Let $w = (w_0, \ldots, w_n)$ be a tuple of weights and denote by $\PP(w)$ the corresponding weighted projective space.
For simplicity we will always assume that the coordinates $w_i$ of $w$ are coprime positive integers, and we write $|w| = \prod_i w_i$.
For such a $w$, a polynomial $f\in \ZZ[X_0, \ldots, X_n]$ is said to be \emph{weighted homogeneous of degree $d$ with weights $w$} if $f(X_0^{w_0}, \ldots, X_n^{w_n})$ is homogeneous of degree $d$.
We call $f$ \emph{full} if $\lcm(w)$ divides $d$.
Note that a weighted homogeneous polynomial $f$ is automatically full if a term involving every variable appears in $f$.
If $f$ is such a polynomial, then the locus $V(f)$ of $f$ is an algebraic hypersurface in $\PP(w)$. We will work with such a fixed tuple $w$ which is sometimes kept implicit.
We will write $m = n-1$ for the dimension of $V(f)$ and write $||f||$ for the maximum of the absolute values of the coefficients of $f$.
For $B$ a positive integer, let $N(f,B)$ be the number of $x = (x_0 : \ldots : x_n)\in \PP(w)(\QQ)$ for which $f(x) = 0$, $|x_i|\leq B^{w_i}$ and $x_i\in\ZZ$ for each $i$.
Also denote by $V(f)(B)$ the set of $x\in \PP(w)(\QQ)$ for which $f(x) = 0$,  $|x_i|\leq B^{w_i}$ and $x_i\in\ZZ$ for each $i$.

In general, the space $\PP(w)$ has finitely many singular points.
For a prime $p$, denote by $\PP(w)_p$ the reduction of $\PP(w)$ modulo $p$, which is a variety over $\FF_p$.
Let $\Sigma\subset \PP(w)(\QQ)$ consist of all points $P$ on $\PP(w)$ for which there exists a prime $p$ such that $P$ becomes singular on $\PP(w)_p$.
We will only count rational and integral points outside of $\Sigma$.

The following two theorems are the main results of this section.
They allow us to catch the rational or integral points on a hypersurface in $\PP(w)$ in an auxiliary hypersurface of controlled degree.

\begin{thm}\label{thm:aux.general}
Assume that $w_n = 1$, let $f\in \ZZ[X_0, \ldots, X_n]$ be a full weighted homogeneous polynomial of degree $d$ with weights $w$ and let $B\geq 2$ be a positive integer.
Then there exists a weighted homogeneous polynomial $g\in \ZZ[X_0, \ldots, X_n]$ coprime to $f$, vanishing on
\[
\{(x_0 : \ldots : x_n)\in V(f)(\QQ)\mid |x_i| \leq B^{w_i}\mbox{ for each }i \}\setminus \Sigma,
\]
and of weighted degree
\[
M \ll_{w} d^{4-\frac{1}{m}} B^{\frac{m+1}{m} \frac{|w|^{1/m}}{d^{1/m}}} \frac{b(f)}{||f||^{\frac{|w|^{1/m}}{md^{1+1/m}}}} + d^2\log (B)  + d^{4-\frac{1}{m}} .
\]
\end{thm}

We restrict to points not in $\Sigma$ for certain determinant estimates.
Note that for a polynomial $F$ as in Equation~\ref{eq:FYX}, we automatically have that $V(F)$ is disjoint from $\Sigma$.
Indeed, this follows from the fact that $F$ is monic in $Y$ and that if $w=(e,1\ldots, 1)$ then $\PP(w)_p$ consists of at most the single point $(1:0:\ldots: 0)$ for every prime $p$.

The quantity $b(f)$ will be defined below, and is related to absolute irreducibility of $f$ modulo primes $p$.
For now let us simply note that $b(f)$ is always bounded by $O_n(\max\{\log ||f|| / d^2, 1\})$.

In the affine situation we have the following.
For $f\in \ZZ[X_1, \ldots, X_n]$ any polynomial, we say that $f$ has \emph{weighted degree $d$ (with weights $w$)} if $f(X_1^{w_1}, \ldots, X^{w_n})$ has degree $d$.
Note that if $f$ has weighted degree $d$ for $w = (w_1, \ldots, w_n)$, then we can homogenize $f$ to a weighted homogeneous $F(X_0, \ldots, X_n)$ of degree $d$ with weights $w' = (1, w_1, \ldots, w_n)$.
The space $\AA^n$ sits naturally inside $\PP(w')$ by considering only points of the form $(1:a_1:\ldots:a_n)$.

We deduce from the above theorem the following affine variant.

\begin{thm}\label{thm:aux.affine.general}
Let $f\in \ZZ[X_1, \ldots, X_n]$ be absolutely irreducible, full, primitive, of weighted degree $d$ with weights $w$ and let $B\geq 2$ be a positive integer.
Write $w' = (1,w)$ and let $\Sigma\subset \PP(w')$ be as constructed above.
Then there exists a polynomial $g\in \ZZ[X_1, \ldots, X_n]$ coprime to $f$, vanishing on
\[
\{(x_1, \ldots , x_n)\in \ZZ^n\mid f(x)=0\mbox{ and } |x_i| \leq B^{w_i}\mbox{ for each }i \}\setminus \Sigma,
\]
and of degree
\[
M \ll_{w} d^{2-\frac{1}{m}}  B^{\frac{|w|^{1/m}}{d^{1/m}}} \frac{\min\{ \log ||f_d|| + d\log B + d^2, d^2b(f) \}}{||f_d||^{\frac{|w|^{1/m}}{md^{1+1/m}}}} + d^{4-\frac{1}{m}}\log B.
\]
\end{thm}

For affine curves with weights $(e,1)$ we obtain the following for counting integral points.

\begin{cor}\label[cor]{cor:weighted.aff.curve}
Let $f\in \ZZ[Y, X]$ be absolutely irreducible, full, of weighted degree $de$ with weights $w=(e,1)$ and monic in $Y$ with $\deg_Y f = de$.
Then
\[
N(f;B^e,B)\ll_e d^4 B^{\frac{1}{d}}\log B.
\]
\end{cor}

\begin{proof}
Let $F(Y,X_0,X_1)$ be the weighted homogenization of $f$, with weights $(e,1,1)$.
Note that the only singular point of $\PP(e,1,1)$ is $[1:0:0]$ and so the fact that $f$ is monic in $Y$ of degree $de$ shows that for every prime $p$ we have that $V(F)_p$ is disjoint from the singular locus of $\PP(e,1,1)_p$.
Now apply Theorem~\ref{thm:aux.affine.general} and use B\'ezout's theorem to conclude.
\end{proof}

As for surfaces, we record the following special case for later use.

\begin{cor}\label[cor]{cor:YX1X2}
Let $F\in \ZZ[Y, X_1, X_2]$ be absolutely irreducible, full, of weighted degree $de$ and with weights $w=(e,1,1)$, for some $d\ge 2$ and $e\ge 1$.
Assume that $F$ is monic in $Y$ with $\deg_Y F = de$.
Then for each $B\ge 2$ there exists a polynomial $g\in \ZZ[Y, X_1, X_2]$ coprime to $F$, vanishing on
\[
\{(y,x)\in \ZZ^3\mid |y| \leq B^e,\ |x_1| \leq B,\ |x_2| \leq B,\  F(y,x)=0\}
\]
and of degree
\[
M \ll_e  d^{3+1/2} B^{\frac{1}{\sqrt{d}}} \log B.
\]
\end{cor}

\begin{proof}
In the same way as the previous proof, the fact that $F$ is monic in $Y$ shows that $V(F)_p$ is disjoint from the singular locus of $\PP(e,1,1,1)$ for every prime $p$.
Then apply \Cref{thm:aux.affine.general} to $f$ to obtain the desired polynomial $g$.
To bound its degree, use that
\[
\frac{\log ||f_d||}{||f_d||^{\frac{e^{1/2}}{m(de)^{3/2}}}} \ll_e d^{3/2}.\qedhere
\]
\end{proof}

\subsection{Some determinant estimates}

For the proof of \Cref{thm:aux.general} we follow the treatment of the $p$-adic determinant as in~\cite{CCDN-dgc}.
This builds upon work of Heath--Brown~\cite{Heath-Brown-Ann}, Walsh~\cite{Walsh}, and especially Salberger~\cite{Salberger-dgc}.

We first recall some determinant estimates.
Let $X = V(f) \subset \PP(w)$ be the subvariety cut out by a weighted homogeneous primitive polynomial $f$ in $\PP(w)$.
Recall that $m = \dim X = n-1$.
For $p$ a prime number, we denote by $X_p$ the reduction of $X$ modulo $p$, which is the variety defined by $f$ over $\FF_p$.

\begin{lem}
Let $p$ be a prime, let $P$ be an $\FF_p$-point on $X_p$ of multiplicity $\mu$, let $\xi_1, \ldots, \xi_s$ be $\ZZ$-points on $X$ reducing to $P$ on $X_p$, and let $F_1, \ldots, F_s\in \ZZ[X]$ be weighted homogeneous. Suppose that $P$ is a smooth point of $\PP(w)_p$.
Then $\det(F_i(\xi_j))_{ij}$ is divisible by $p^e$, where
\[
e\geq \left(\frac{m!}{\mu}\right)^{1/m} \frac{m}{m+1} s^{1+\frac{1}{m}} - O_n(s).
\]
\end{lem}

\begin{proof}
The proof is identical~\cite[Lem.\,2.5]{Salberger-dgc} and~\cite[Cor.\,2.5]{CCDN-dgc}.
Indeed, note that the proof given there is local.
\end{proof}

\begin{lem}
Let $p$ be a prime, let $F_1, \ldots, F_s\in \ZZ[X]$ be weighted homogeneous and let $\xi_1, \ldots, \xi_s$ be $\ZZ$-points on $X$. Suppose that the reductions of  the $\xi_i$ modulo $p$  are smooth on $\PP(w)_p$.
Then $\det(F_i(\xi_j))$ is divisible by $p^e$ with
\[
e\geq m!^{1/m}\frac{m}{m+1} \frac{s^{1+\frac{1}{m}}}{n_p^{1/m}} - O_n(s),
\]
where $n_p$ is the number of $\FF_p$-points on $X_p$, counted with multiplicity.
\end{lem}
\begin{proof}
Again, this is identical to~\cite[Lem.\,1.4]{Salberger-dgc} or~\cite[Prop.\,2.6]{CCDN-dgc}.
\end{proof}

Assume that $p$ is such that $X_p$ is absolutely irreducible and $p>27d^4$.
Let $Y\subset \AA^{n+1}$ be the affine cone over $X$.
Then~\cite[Cor.\,5.6]{CafureMatera} shows that $Y_p$ has
\[
p^{m+1} + O_m(d^2p^{m+\frac{1}{2}})
\]
points over $\FF_p$, counted without multiplicity.
Since the singular points are contained in a codimension at least $1$ subvariety of $Y$ of degree at most $d^2$, the Lang--Weil bound shows that $Y_p$ in fact has $p^{m+1} + O_m(d^2p^{m+\frac{1}{2}})$ points over $\FF_p$ even when counted with multiplicity.
Let $\lambda, \mu\in \FF_p^\times$ and assume that $\lambda^{w_i} = \mu^{w_i}$ for $i=0, \ldots, n$.
Since the $w_i$ are coprime by assumption we then obtain that $\lambda = \mu$.
Therefore we conclude that $X_p$ has at most
\[
p^m + O_m(d^2 p^{m-\frac{1}{2}})
\]
points over $\FF_p$.

\begin{defn}
Define
\[
b(f) = \prod_p \exp\left(\frac{\log p}{p} \right)
\]
where the product is over all primes $p$ such that $p>27d^4$ and such that $f\bmod p$ is not absolutely irreducible.
\end{defn}

\begin{lem}\label[lem]{lem:b(f).bound}
We have
\[
b(f)\ll_n \max\left\{ \frac{\log ||f||}{d^2}, 1\right\}.
\]
\end{lem}

\begin{proof}
Forgetting the weights $w$ for a moment, we can homogenize $f$ to a polynomial $F\in \ZZ[X_0, \ldots, X_n, Y]$ which is homogeneous of degree at most $d$.
Then $f$ is absolutely irreducible if and only if $F$ is.
Hence the result follows from~\cite[Cor.\,3.2.3]{CCDN-dgc} applied to $F$.
\end{proof}

\begin{prop}\label[prop]{prop:det.est}
Let $\xi_1, \ldots, \xi_s\in X(\QQ)\setminus \Sigma$ and $F_{i\ell}\in \ZZ[X_0, \ldots, X_n]$ weighted homogeneous for $i=1, \ldots, s, \ell=1, \ldots, L$.
Put $\Delta_\ell = \det(F_{i\ell}(\xi_j))_{ij}$ and $\Delta = \gcd (\Delta_\ell)_\ell$.
If $\Delta\neq 0$ then
\[
\log |\Delta| \geq \frac{m!^{1/m}}{m+1} s^{1+\frac{1}{m}}\left(\log s - O_n(1) - m(4\log d + \log b(f))\right).
\]
\end{prop}

\begin{proof}
The proof is exactly as in~\cite[Prop.\,3.2.6]{CCDN-dgc}.
\end{proof}

\subsection{A coordinate transformation}

We need a coordinate transformation to be able to assume that $f$ has a certain large coefficient.
For this we will assume that $w_n = 1$, which is enough for our purposes.
For $f\in \ZZ[X_0, \ldots, X_n]$ weighted homogeneous of degree $d$, denote by $c_f$ the coefficient of $X_n^d$ in $f$.

\begin{lem}
Let $f\in \ZZ[X_0, \ldots, X_n]$ be weighted homogeneous of degree $d$.
Then there exist $a_0, \ldots, a_{n-1}\in \{0, \ldots, d\}$ such that
\[
|f(a_0, \ldots, a_{n-1}, 1)| \geq ||f|| / 3^{nd}.
\]
\end{lem}

\begin{proof}
This follows from~\cite[Lem.\,3.4.2]{CCDN-dgc}.
\end{proof}

\begin{lem}\label[lem]{lem:coord.transf}
Assume that $w_n = 1$.
There exists a constant $C = O_n(1)$ such that the following holds.
Let $f\in \ZZ[X_0, \ldots, X_n]$ be primitive and weighted homogeneous of degree $d$.
Then there exists $g\in \ZZ[X_0, \ldots, X_n]$ primitive and weighted homogeneous of degree $d$ with the following properties:
\begin{enumerate}[label=(\roman*)]
\item $|c_g| \geq ||g|| C^{-d \log d}$,
\item $||g|| \leq C^{d\log d} ||f||$, and
\item for any positive integer $B\geq 2$ we have
\[
N(f,B) \leq N(g, (d+1)B).
\]
\end{enumerate}
\end{lem}

\begin{proof}
By the previous lemma, there exist $a_0, \ldots, a_{n-1}\in \{0, \ldots, d\}$ such that
$$
|f(a_0, \ldots, a_{n-1}, 1)| \geq ||f|| / 3^{nd}.
$$
Now simply define
\[
g(x_0, \ldots, x_n) = f(x_0 + a_0x_n^{w_1}, \ldots, x_{n-1} + a_{n-1}x_n^{w_{n-1}}, x_n).\qedhere
\]
\end{proof}

\subsection{Proof of Theorem~\ref{thm:aux.general}}

We need a lemma about counting the number of weighted monomials of a given degree.
Equivalently, this lemma is about counting integral points on certain lattices.

\begin{lem}\label[lem]{lem:count.mon}
For a positive integer $M$ let $b_w(M)$ be the number of $(e_0, \ldots, e_n)\in \NN^{n+1}_{\geq 0}$ for which
\[
\sum_i w_i e_i = M.
\]
If $M$ is divisible by $\operatorname{lcm}(w)$ then there exists a constant $c=O_w(1)$  such that
\[
b_w(M) = \frac{1}{|w|}\binom{M+n}{n} + c\binom{M + n - 1}{n - 1} + O_w(M^{n - 2}) .
\]
\end{lem}

\begin{proof}
Consider the generating function $F(z) := \sum_{M} b_w(M) z^M = \prod_i \frac{1}{1-z^{w_i}}$. Note that it has a meromorphic continuation to the entire complex plane except for possible poles at $\lcm_i(w_i)$-th roots of unity.

The function $F(z)$ has a pole of order ${n + 1}$ only at $z = 1$ because $\gcd_i(w_i) = 1$. Moreover we have
\[
F(z)(1 - z)^{n + 1}|_{z = 1} = \prod_i \frac{1}{1 + z + \cdots + z^{w_i - 1}}\Big|_{z = 1} = \prod_{i} \frac{1}{w_i}.
\]

Let $\Xi$ consist of all roots of unity where $F(z)$ has a pole of at least order $n$. Then \cite[Thm.~5.2.1 and (5.2.5)]{Wilf2006Generating} implies that for each $\zeta \in \Xi$ there exists a $c_{\zeta}$ such that
\[
b_w(M) = \prod_{i} \frac{1}{w_i} \binom{M + n}{n} + \sum_{\zeta \in \Xi} c_{\zeta} \zeta^{M + n}\binom{M + n - 1}{n - 1} + O_w(M^{n - 2}).
\]
The lemma follows since $\zeta^M = 1$ for all $\zeta \in \Xi$ if $\lcm_i(w_i) \mid M$ by the definition of $\Xi$.
\end{proof}

We can now prove Theorem~\ref{thm:aux.general}.

\begin{proof}[Proof of Theorem~\ref{thm:aux.general}]
Let $S\subset \PP(w)(\QQ)$ be the set of rational points on $X$ of height at most $B$ which do not lie on $\Sigma$.
Let $M$ be a large positive integer.
We will find the desired $g$ of weighted degree $M$.
We can already assume that $M$ is larger than $O(d^2)$ and that $M$ is divisible by $\lcm(w)$.
Assume towards a contradiction that no such $g$ exists.
We will show that $M$ is bounded as stated.
Using Lemma~\ref{lem:coord.transf} and the fact that $w_n = 1$ we may assume that
\[
|c_f| \geq ||f|| / C^{d\log d}
\]
for some $C = O_n(1)$, at the cost of replacing $B$ by $B' = (d+1)B$.

For $e$ a positive integer, let $\cB[e]$ be the set of weighted homogeneous monomials of degree $e$.
By \Cref{lem:count.mon} we have
\[
|\cB[M]| = \frac{1}{|w|} \binom{M+m+1}{m+1} + O_w(M^m).
\]
Let $\Xi\subset S$ be a subset which is maximally algebraically independent over $\cB[M]$, and write $s = |\Xi|, r = |\cB[M]|$.
Let $A$ be the $s\times r$ matrix $(F(\xi))_{\xi\in \Xi, F\in \cB[M]}$, which is of full rank $s$.
Every entry of $A$ is bounded in absolute value by $B'^M$.
By assumption, every polynomial of weighted degree $M$ vanishing on $\Xi$ is in the span of $f\cdot \cB[M-d]$.
Hence 
\[
s= |\cB[M]| - |\cB[M-d]|.
\]

Let $\Delta$ be the greatest common divisor of all $s\times s$ minors of $A$.
Since any element of $f\cdot \cB[M]$ has a coefficient of size at least $|c_f|$, the Bombieri--Vaaler theorem from Theorem~\ref{thm:BV} shows that
\[
|\Delta| \leq \sqrt{|\det(AA^T)|} (||f|| / C^{d\log d})^{s-r}.
\]
Since $\Delta\neq 0$, we can combine this with the determinant estimate from Proposition~\ref{prop:det.est} to obtain that
\begin{align*}
&\frac{m!^{1/m}}{m+1}s^{1+1/m} (\log s - O_n(1) - 4m\log d - m\log b(f)) \\
\leq &\frac{\log s!}{2} + \frac{s \log r}{2} + sM\log B' - (r-s)(\log ||f|| -O_n(d\log d)).
\end{align*}
By Stirling's approximation $\log(s!)\leq s\log s$ and $\log r \leq \log (M+1)^{m+1} = O_n(\log M) = O_n(\log s)$.
Hence the first two terms on the right-hand-side may be absorbed in the left-hand-side by adjusting the $O_n(1)$.
Divide by $Ms$ to get
\begin{align*}
& \frac{m!^{1/m}}{m+1} \frac{s^{1/m}}{M}\left(\log s - O_n(1) - 4m\log d - m\log b(f)\right) \\
&\leq \log B' - \frac{r-s}{Ms} (\log ||f|| - O_n(d\log d)).
\end{align*}
By Lemma~\ref{lem:count.mon} and using that $M$ is divisible by $\lcm(w)$ we have that
\[
s = \frac{dM^m}{|w|m!} + O_w(d^2M^{m-1}).
\]
Therefore $\log s = \log d + m\log M - O_w(1)$.
Since $M$ is larger than some chosen multiple of $d^2$, we also obtain
\[
\frac{s^{1/m}}{M} = \frac{d^{1/m}}{|w|^{1/m} m!^{1/m}} + O_w(d^2 / M).
\]
Because $f$ is full, $d$ is divisible by $\lcm(w)$ and hence so is $M-d$.
Therefore, Lemma~\ref{lem:count.mon} yields that
\[
\frac{r-s}{Ms} = \frac{1}{d(m+1)} + O_w(1/M).
\]
Putting this together, the inequality becomes
\begin{align*}
& \frac{m d^{1/m}}{(m+1) |w|^{1/m}} \left( \log M - O_w(1) - (4-1/m)\log d - (1+O_n(d^{2-1/m}/M))\log b(f)\right)\\
& \leq \log B' - \frac{\log ||f||}{d(m+1)} + O_w(\log ||f|| / M) + O_w(\log d) + O_w(d\log d / M).
\end{align*}
The terms $O_w(\log d)$ and $O_w(d\log d / M)$ on the right-hand-side may be absorbed in the $O_w(1)$ on the left-hand-side.
Also, since $\log B' = \log B + \log (d+1)$, we may replace $\log B'$ by $\log B$ in the above inequality.
So we are left with
\begin{align*}
& \frac{m d^{1/m}}{(m+1) |w|^{1/m}} \left( \log M - O_w(1) - (4-1/m)\log d - (1+O_n(d^{2-1/m}/M))\log b(f)\right)\\
&\leq \log B - \frac{\log ||f||}{d(m+1)} + O_w(\log ||f|| / M).
\end{align*}

Assume first that $||f|| \leq B^{2d(m+1)}$ and $M\geq d^{1-1/m}\log B$.
Then $\log ||f|| \leq 2d(m+1)\log B = O_n(d^{1/m}M )$ and so we can drop the last term on the right-hand-side by adjusting the $O_w(1)$ on the left-hand-side.
Also, we can drop the term $O_n(d^{2-1/m}/M)\log b(f)$ on the left-hand-side by using Lemma~\ref{lem:b(f).bound}.
We thus end up with
\[
\log M \leq O_w(1) + \frac{(m+1)|w|^{1/m}}{md^{1/m}}\log B - \frac{|w|^{1/m}\log ||f||}{md^{1+1/m}} + (4-1/m)\log d + \log b(f).
\]
Secondly assume that $||f||\geq B^{2d(m+1)}$ and that $M\gg_w d$ so that in the above equation we have that
\[
-\frac{\log ||f||}{d(m+1)} + O_w(\log ||f|| / M) \leq -\frac{\log ||f||}{2d(m+1)}.
\]
Then rearranging the above inequality we obtain
\[
\log M \leq O_w(1) + (4-1/m)\log d + (1+O_w(d^{-1/m}))\log b(f) - \frac{|w|^{1/m}\log ||f||}{2md^{1+1/m}}.
\]
By using Lemma~\ref{lem:b(f).bound} we see that
\[
(1+O_w(d^{-1/m}))\log b(f) - \frac{|w|^{1/m}\log ||f||}{2md^{1+1/m}} = O_w(1),
\]
and so the result follows.
\end{proof}

For the proof of \Cref{thm:aux.affine.general} we use a technique due to Ellenberg--Venkatesh \cite{Ellenb-Venkatesh}.

\begin{proof}[Proof of \Cref{thm:aux.affine.general}]
We can write $f = \sum_{i=0}^d f_i$ where $f_i\in \ZZ[X_1, \ldots, X_n]$ is weighted homogeneous of weighted degree exactly $i$.
For $H$ a positive integer we define
\[
F_H = \sum_{i=0}^d H^i f_i X_0^{d-i}\in \ZZ[X_0, \ldots, X_n].
\]
We give $X_0$ weight $1$, so that $F_H$ is weighted homogeneous of degree $d$.
Note also that $F_H$ is still full and absolutely irreducible, but not necessarily primitive.

Let $B'$ be a prime in $(B/2, B]$.
If $B'$ is not a divisor of $f_0$ then $F_{B'}$ is primitive, and so by Theorem~\ref{thm:aux.general} there exists a weighted homogeneous $G\in \ZZ[X_0, \ldots, X_n]$ of weighted degree
\[
\ll_w d^{4-\frac{1}{m}} B^{\frac{m+1}{m}\frac{|w|^{1/m}}{d^{1/m}}} \frac{b(F_{B'})}{||F_{B'}||^{\frac{|w|^{1/m}}{md^{1+1/m}}}} + d^2 \log B   + d^{4-\frac{1}{m}}
\]
vanishing on $V(F_{B'})(B)$.

We have $||F_{B'}||\geq 2^{-d} B^d ||f_d||$ and so by~\cite[Lem.\,4.2.3]{CCDN-dgc} we have
\[
d^{4-\frac{1}{m}} B^{\frac{m+1}{m}\frac{|w|^{1/m}}{d^{1/m}}} \frac{b(F_{B'})}{||F_{B'}||^{\frac{|w|^{1/m}}{md^{1+1/m}}}} \ll_w d^{2-\frac{1}{m}} B^{\frac{|w|^{1/m}}{d^{1/m}}} \frac{\log ||f_d|| + d\log B + d^2}{||f_d||^{\frac{|w|^{1/m}}{md^{1+1/m}}}}.
\]
Also $b(F_{B'})$ agrees with $b(F_1) = b(f)$ up to a factor $O(1)$.
The set of rational points of $\PP(w')$ of height at most $B$ which reduce to a singular point on $\PP(w')_{B'}$ has size $O_w(1)$.
Hence there exists a polynomial $h(x_0, \ldots, x_n)$ of degree at most $O_w(1)$ vanishing on these which is coprime to $f$.
Then $g(x_1, \ldots, x_n) = h(1, x_1, \ldots, x_n)\cdot G(H, x_1, \ldots, x_n)$ is as desired.

By~\cite[Lem.\,5]{Brow-Heath-Salb} we may assume that $|f_0| = O_n(B^{d^{n+2}})$.
Now assume that $B'\mid f_0$ for all primes $B'\in (B/2, B]$.
Then
\[
\left(\prod_{\substack{B'\in (B/2, B] \\ B' \text{ prime}}} B'\right)\mid f_0.
\]
So if $f_0$ is non-zero then
\[
\sum_{\substack{B'\in (B/2, B] \\ B' \text{ prime}}} \log B' \leq \log |f_0| \ll_n d^{n+2} \log B.
\]
By the prime number theorem the left-hand-side is bounded below by $cB$ for some $c = O(1)$.
Hence $B$ is bounded by a polynomial in $d$ (which depends on $n$), and we have $B^{\frac{|w|^{1/m}}{d^{1/m}}} = O_w(1)$.
By \Cref{thm:aux.general}, there exists a weighted homogeneous polynomial $G\in \ZZ[X_0, \ldots, X_n]$ coprime to $F_1$, vanishing on the points of $F_1$ of height at most $B$, and of degree at most
\[
\ll_w d^{4-\frac{1}{m}} B^{\frac{m+1}{m}\frac{|w|^{1/m}}{d^{1/m}}} \frac{b(F_1)}{||F_1||^{\frac{|w|^{1/m}}{md^{1+1/m}}}} + d^2 \log B  + d^{4-\frac{1}{m}} .
\]
Now $||F_1|| \geq ||f_d||$ and $b(F_1) = b(f)$.
By~\cite[Lem.\,4.2.3]{CCDN-dgc} we have
\[
d^{4-\frac{1}{m}} \frac{b(F_1)}{||F_1||^{\frac{|w|^{1/m}}{md^{1+1/m}}}} \ll_w d^{2-\frac{1}{m}}\frac{\min\{d^2b(f), \log ||f|| + d^2\}}{||f_d||^{\frac{|w|^{1/m}}{md^{1+1/m}}}}
\]
and so $G(1, x_1, \ldots, x_n)$ is as desired.

Finally, if $f_0 = 0$ then by~\cite[Cor.\,4.1.2]{CCDN-dgc} there exist $a_1, \ldots, a_n\in \ZZ$ with $|a_i|\leq d$ such that $f_0(a_1, \ldots, a_n)\neq 0$.
Then we apply the above argument to $\tilde{f}(x) := f(x_1+a_1, \ldots, x_n+a_n)$ with $\tilde{B} = B+d$ to conclude. 
\end{proof}

\section{Counting on twisted lines}\label{sec:twisted}

We will prove \Cref{thm:mainFYX:unif} by induction on the dimension, where the key case is that of surfaces.
On surfaces we need to count integral points separately on a certain notion of twisted lines, defined in \Cref{defn:tw-l}.
The main result of this section is \Cref{prop:tw-l}, which gives a strong enough upper bound for counting integral points on these twisted lines.

Let $F\in \ZZ[Y, X_1, X_2]$ be as in (\ref{eq:GYX}) of the introduction, with $n=2$. So, in particular we have $F = F_{\mathrm{top}} + F_0$ such that $F_{\mathrm{top}}(Y^e, X_1, X_2)$ is homogeneous of degree $de$ and $F_{0}(Y^e, X_1, X_2)$ is a polynomial of degree less than $de$. Assume in this whole section that $F_{\mathrm{top}}$ is absolutely irreducible. For such $F$ let  $X$ be $V(F) \subset \AA^3$ and let $\pi: \AA^3\to \AA^2$ be the projection onto the $(X_1,X_2)$ coordinates.

\begin{defn}\label[defn]{defn:tw-l}
A \emph{twisted line on $X$ over $\QQ$} is a subvariety $\ell\subset X$, defined over $\QQ$, such that $\pi(\ell)$ is a line in $\AA^2$ and the restriction of $\pi$ to $\ell$ is an isomorphism onto its image.
\end{defn}

We have to count integral points separately on these twisted lines, similar to \cite[Prop.\,1]{Brow-Heath-Salb} and \cite[Prop.\,4.3.3]{CCDN-dgc}.
We first prove some lemmas.

\begin{lem}\label[lem]{lem:tw-l}
Let $\ell\subset X$ be a twisted line over $\QQ$.
Then there exist rational numbers $a_0$, $a_1$, $a_2$, $v_1$, $v_2$ $w_1, \ldots, w_e$ such that
\[
\ell(\QQ) = \{(a_0 + tw_1 + t^2w_2 + \ldots + t^e w_e, a_1 + tv_1, a_2 + tv_2)\mid t\in \QQ\}.
\]
Furthermore, we can take $(w_e, v_1,v_2)$ to be a triple of integers with $\gcd(v_1,v_2)=1$.
\end{lem}

\begin{proof}
Substitute the equation for the line $\pi(\ell)$ and factor $F$ and let $T$ be the new coordinate for the line.
The fact that $\pi$ restricted to $\ell$ is an isomorphism shows that this is described by an equation of degree $1$ in $Y$ and at most $e$ in $T$. That $(v_1,v_2)\neq (0,0)$ follows from the fact that the fibers of $\pi$ are finite.

Finally, we show that if the $v_i$ are integers, then so is $w_e$.
Indeed, the $w_i$, $a_i$, $v_i$ satisfy a system of equations determined by
\[
F(a_0+...+w_e t^e, a_1+tv_1, a_2+tv_2)=0
\]
for all $t$.
The coefficient of $t^{ed}$ in the above expression is given by $F_{\rm top}(w_e, v_1, v_2)$.
The specific form of $F$ then shows that $w_e$ is integral over $\ZZ$.
As $w_e$ is rational, we conclude that $w_e \in \ZZ$, as desired.
\end{proof}

\begin{lem}\label[lem]{lem:tw-count}
Let $\ell\subset \AA^3$ be a twisted line on $X$ over $\QQ$ and let $a_i, w_i, v_i$ be as in \Cref{lem:tw-l}.
Then, assuming the $v_i$ are coprime integers and hence $w_e$ is an integer, we have

\[
N_{\rm aff} (\ell; B^e, B, B) \le \frac{2e B}{\max\{|w_e|^{1/e}, |v_1|, |v_2|\}}+e.
\]
\end{lem}

\begin{proof}
Note that if there is at least one triple of integers $(a, b, c)$ lying on $\ell$, then we may assume that $(a_1, a_2)=(b,c)$ and are hence integers.
In that case only integer values of $t$ in the description from \Cref{lem:tw-l} can contribute to $N_{\rm aff} (\ell; B^e, B, B)$. 	
	
The only non-obvious thing left to show then is that there are at most $2eB/|w_e|^{1/e}+e$ integers $t$ for which $|a_0+...+w_e t^e|\le B^e$.
For this, let $\alpha_1, ..., \alpha_e\in \CC$ be the roots of the polynomial $a_0+...+w_et^e$.
Then we find that
\[
|t-\alpha_1|\ldots|t-\alpha_e| \le \frac{B^e}{|w_e|}.
\]
But if this holds, then certainly $|t-\alpha_i| \le B/|w_e|^{1/e}$ for some $i$.
As there are no more than $2B/|w_e|^{1/e}+1$ such integers $t$ for each $\alpha_i$, the result follows.
\end{proof}

Let $X_\infty$ be the curve defined by $F_{\mathrm{top}} = 0$ in $\PP(e,1,1)$.
Recall that by assumption this curve is absolutely irreducible.
If $\ell$ is a twisted line on $X$ (over $\QQ$) then it is described as in this lemma above, and we call $(w_e:v_1:v_2)$ the \emph{direction of $\ell$}.
Note that $(w_e:v_1:v_2)$ is a point on $X_\infty$.

\begin{lem}\label[lem]{lem:tw-l:few}
Let $(w_e:v_1:v_2)\in X_\infty$.
Then there are at most $(de)^{e+1}$ twisted lines on $X$ over $\QQ$ with direction $(w_e:v_1:v_2)$.
\end{lem}

\begin{proof}
We will first show that the number of twisted lines on $X$ with a given direction is finite.
By then considering explicit equations, we can effectively bound the number of twisted lines on $X$ with a given direction.
We denote by $\overline{X}$ the closure of $X$ inside $\PP(e,1,1,1)$.

We start by constructing a proper variety parametrizing twisted lines in $\PP(e,1,1,1)$.
Let $\cX = \PP(e,e-1,\ldots, 2,1,1,1,1)\times \PP(e,1,1,1)$, which is proper.
For a point $\xi=((w_e: w_{e-1}: \ldots: w_2: w_1: v_1: v_2:v_3), (a_0: a_1: a_2: a_3))\in \cX$ we define $\dire \xi= (w_e /v_3^e : v_1 / v_3 : v_2 / v_3)\in \PP(e,1,1)$ to be the \emph{direction} of $\xi$, if it is well-defined.
Consider the regular map $\phi_\xi$ given by
\[
\PP^1\to \PP(e,1,1,1): (t:s)\mapsto (a_0s^e + w_1s^{e-1}t + \ldots + w_et^e: a_1s+tv_1 : a_2s+tv_2 : a_3s + tv_3).
\]
If $(v_1,v_2)\neq (0,0)$ then the image of $\phi_\xi$ is a twisted line in $\PP(e,1,1,1)$, and we will identify points $\xi$ with the image of the map $\phi_\xi$.
Note however, that different $\xi$ can yield the same twisted line.
Also, if $(v_1,v_2) = (0,0)$ then the image of $\phi_\xi$ is not a twisted line, and is instead a vertical line.

Assume that $X$ contains infinitely many twisted lines with direction $(w_e:v_1:v_2)\in \PP(e,1,1)$.
Define
\[
\cY = \{(\xi, p)\in \cX\times \overline{X}\mid p\in \phi_\xi(\PP^1)\subset \overline{X}\text{ and }\dire \xi = (w_e:v_1:v_2)\},
\]
and denote by $\pi: \cY\to \overline{X}$ the projection map.
Clearly, $\cY$ is a closed subvariety of $\cX\times \overline{X}$.
Therefore, also $\pi(\cY)$ is a closed subvariety of $\overline{X}$.
Now, $\pi(\cY)$ is exactly the union of all twisted lines on $\overline{X}$ with direction $(w_e:v_1:v_2)$, which we assumed to be of dimension $2$.
Since $\overline{X}$ is geometrically integral, we hence obtain that $\pi(\cY) = \overline{X}$.
In other words, through every point of $\overline{X}$, there is a twisted line on $\overline{X}$ passing through $(w_e:v_1:v_2:0)\in X_\infty \subset \overline{X}$.
Take any point $x\in X_\infty$ not equal to $(w_e:v_1:v_2)$.
Then there exists a twisted line on $\overline{X}$ passing through $x$ and $(w_e:v_1:v_2:0)$, which must be completely contained in $X_\infty$.
But this contradicts the fact that $X_\infty$ is geometrically integral.
So we indeed conclude that there are only finitely many twisted lines on $X$ with a given direction.

Finally let us count the maximal number of such twisted lines.
Consider the equation
\[
F(a_0+w_1t+...+w_e t^e, a_1+tv_1, a_2+tv_2)=0
\]
in the variables $a_0, a_1, a_2$ and $w_1, ..., w_{e-1}$.
We may assume that $a_1=0$ or $a_2=0$ depending on whether $v_1 \neq 0$ or $v_2 \neq 0$.
After this change of variables, the other coefficients are uniquely determined by the twisted line.
We then want to show that there are at most $(de)^{e+1}$ solutions.
By the preceding, the number of solutions is finite.
There are $e+1$ variables determined by $de$ equations, all of degree at most $de$, whence there are at most $(de)^{e+1}$ of them, as desired.
\end{proof}

\begin{lem}\label[lem]{lem:tw-l:rat}
There are at most $O(e^2 d^3 B^{2/d})$ many points $(w_e:v_1:v_2)$ of weighted height at most $B$ lying on $X_\infty$.
\end{lem}
\begin{proof}
This follows from Theorem \ref{thm:aux.general} and Lemma \ref{lem:b(f).bound}, using the inequality
\[
\frac{\log ||f||}{||f||^{1/(ed^2)}} \le e d^2 \exp(-1)
\]
and noting that $e^2 d^2 \log(B) \le e^2 d^3 B^{2/d}$.
\end{proof}

\begin{prop}\label[prop]{prop:tw-l}
Let $I$ be a finite collection of twisted lines on $X$, defined over $\QQ$.
Then
\[
N_{\mathrm{aff}}(\cup_{\ell\in I} \ell, B^e, B, B) \ll (ed)^{e+4} B+e \# I
\]
if $d>2$. If $d=2$ then one has
\[
N_{\mathrm{aff}}(\cup_{\ell\in I} \ell, B^e, B, B) \ll (ed)^{e+4} B \log(B) + e\# I.
\]
\end{prop}
\begin{proof}[Proof of \Cref{prop:tw-l}]
Fix a twisted line $\ell$ in $I$ and write it as above in the form
\[
\ell(\QQ) = \{(a_0 + tw_1 + t^2w_2 + \ldots + t^e w_e, a_1 + tv_1, a_2 + tv_2)\mid t\in \QQ\},
\]
where $v_1, v_2$ are coprime integers and $w_e$ is an integer.
Then by \Cref{lem:tw-count} we have
\[
\# \ell(B^e, B, B)\ll \frac{2e\gcd (w_e, v_1, v_2)B}{\max\{|w_e|^{1/e}, |v_1|, |v_2|\}} + e.
\]
We can ignore the second term by absorbing it in $e \# I$.
Similarly we may assume that every twisted line in $I$ contains at least $e+1$ integral points, otherwise absorbing the count in the term $e \# I$.

We claim that if $\ell$ has at least $e+1$ integral points of weighted height at most $B$, then $\max\{|w_e|^{1/e}, |v_1|, |v_2|\} \le 2B$.
For $v_1$, $v_2$ this is clear, so let us prove the bound on $|w_e|^{1/e}$.
Fix $v_1, v_2, a_1, a_2$ such that all the integral points on $\ell$ are achieved for integral $t$.
Since $\ell(\QQ)$ contains at least $e+1$ points, we can find distinct integers $t_1, \ldots, t_{e+1}$ such that
\[
y_i = a_0 + w_1 t_i + \ldots + w_e t_i^e
\]
are all integers with $|y_i|\leq B^e$, for $i=1, \ldots, e+1$.
Then the Lagrange interpolation formula tells us that
\[
|w_e|=\left|\sum_{i=1}^{e+1} \frac{y_i}{(t_1-t_i)(t_2-t_i)...(t_{i-1}-t_i)(t_{i+1}-t_i)...(t_{e+1}-t_i)}\right| \le \frac{2(e+1) B^e}{e}.
\]
The last inequality follows as among the $t_j$ there is certainly one for which $|t_j-t_i| \ge e/2$.
Hence $|w_e|^{1/e} \le 2B$.

We can then, for $1 \le i \le 2B$, denote by $n_i$ the number of $(w_e:v_1:v_2) \in X_\infty$ for which $i \le \max\{|w_e|^{1/e}, |v_1|, |v_2|\} <i+1$.
As there are at most $(de)^{e+1}$ twisted lines on $X$ of a given direction, to conclude the proof it suffices to prove that
\[
\sum_{i=1}^{2B} \frac{n_i}{i} \ll e^2 d^3+ e^2 d^3 B^{2/d-1} \log B.
\]
Note that by Lemma \ref{lem:tw-l:rat}, we get that
\[
\sum_{i=1}^k n_i \ll e^2 d^3k^{2/d}+ e^2 d^2 \log B.
\]
Using summation by parts we then find that
\begin{align*}
	\sum_{i=1}^{2B} \frac{n_i}{i}&=\sum_{k=1}^{2B-1} \sum_{i=1}^k n_i(\frac1k-\frac{1}{k+1})+\frac{1}{2B} \sum_{k=1}^{2B} n_k\\
	&=\sum_{k=1}^{2B-1} \frac{1}{k^2+k}\sum_{i=1}^k n_i+\frac{1}{2B} \sum_{k=1}^{2B} n_k.
\end{align*}
Now note that $\frac{1}{2B} \sum_{k=1}^{2B} n_k \ll e^2 d^3.$
Finally, note that if $d>2$ we get
\[
\sum_{k=1}^\infty \frac{1}{k^2+k} \sum_{i=1}^k n_i \ll e^2d^3 \sum_{k=1}^\infty \frac{k^{2/3}}{k^2+k} \ll e^2d^3
\]
and if $d=2$ we get
\[
\sum_{k=1}^{2B} \frac{1}{k^2+k} \sum_{i=1}^k n_i \ll e^2d^3 \sum_{k=1}^{2B} \frac{1}{k+1} \ll e^2d^3 \log (B).\qedhere
\]
\end{proof}

\section{Proof of \Cref{thm:mainFYX,thm:main-Serre,thm:mainFYX:unif}}\label{sec:main-proofs}

In this section we prove \Cref{thm:mainFYX,thm:main-Serre,thm:mainFYX:unif}, which rely on the existence of auxiliary polynomials constructed in \Cref{sec:auxi;pol}, more precisely on \Cref{cor:YX1X2}.

\begin{proof}[Proof of \Cref{thm:mainFYX}]
Clearly there is a constant $c>1$ depending only on $n$, $d$, $e$ and $||F||$ such that for all $B\ge 1$ one has
\begin{equation}\label{eqNcoverN}
N_{\rm cover}^{\rm aff}(F, B) \ll N_{\rm aff}(F; c B^e,B,\ldots,B).
\end{equation}
In fact one may take $c = ||F||\binom{n+1+e}{e}$.
Hence, up to replacing $B$ by $cB$, it is sufficient to prove \Cref{thm:mainFYX:unif}.
\end{proof}

\begin{remark}
The above proof is the only place where a dependence on $||F||$ appears in our proof of \Cref{thm:mainFYX}. It is clear from this proof that the dependence on $||F||$ in (\ref{eq:NFB-Serre}) is polynomial, with exponent depending only on $d$, $e$, $n$.
For curves, one can get around this using a clever coordinate transformation to increase the degree of $F$, see~\cite[(2.3)]{Walkowiak} or \cite[Lem.\,4.3]{Verm:aff} for more details.
Such a coordinate transformation crucially relies on the much stronger bounds one has for counting integral or rational points on curves, which becomes better as the degree of the curve increases.
As such, it is not clear to us how to eliminate the dependence on $||F||$ in \Cref{thm:mainFYX}.
\end{remark}

\begin{proof}[Proof of \Cref{thm:mainFYX:unif}]
We proceed by induction on $n$, where the base of the induction is $n=2$.

Fix $B\ge 2$. By \Cref{cor:YX1X2} there exists a polynomial $g(Y,X_1,X_2)$ 
of degree at most
$$
\ll_e d^4 B^{\frac{1}{\sqrt{d}}} \log B,
$$
coprime to $f$, and such that $g$ vanishes on all points in the set
$$
\{(y,x)\mid x \in  [-B, B]^2 \cap \ZZ^2,\ y\in [-B^e, B^e] \cap \ZZ : F(y, x) = 0\}.
$$

Consider the irreducible components $C_j$ (defined over $\QQ$) of the curve $F=g=0$ in $\AA^3$.
Project $C_j$ to the $(X_1,X_2)$-plane and call the projection $D_j$ for each $j$. Then $D_j$ is irreducible of dimension $1$, by the form of $F$. Indeed, $F=0$ cannot contain vertical lines above the $(X_1, X_2)$-plane. The degree $\delta_j$ of $D_j$ is smaller or equal to the degree 
of $C_j$, by basic intersection theory.

If $\delta_j=1$, then the degree $\epsilon_j$ of $C_j$ is at most $de$, by the form of $F$.
Let $aX_1+bX_2=c$ be the equation for the line $D_j$, for some $a,b,c\in \QQ$.
We can assume without loss of generality that $b=1$. 
By factoring $F(Y, X_1, -aX_1-c)$ into absolutely irreducible factors, we obtain an equation of $C_j$ of the form
\[
G(Y,X_1) = Y^k + g_1(X_1)Y^{k-1} + \ldots + g_k(X_1)
\]
of weighted degree $ek$.
If $k=1$ then $C_j$ is a twisted line, and these are taken care of by \Cref{prop:tw-l}.
Otherwise $k\geq 2$ and \Cref{cor:weighted.aff.curve} shows that there at most
\[
\ll_e k^4 B^{1/2}\log B
\]
integral points on $C_j$ that contribute to $N_{\rm aff}(F; B^e, B, B)$.
The total contribution coming from all curves $C_j$ with $\delta_j = 1$ is now bounded by
\[
\ll_e d^8 B^{1/2 + \frac{1}{\sqrt{d}}} \log (B)^2 + d^{e+4}B.
\]
For $d\geq 5$ this is $O_{d,e}(B)$, while for $d=4$ this quantity is $O_{d,e}(B\log(B)^2)$.
For $d=2,3$ this is bounded by $O_{d,e}(B^{1/2+1/\sqrt{d}}\log (B)^2)$, and $1/2+1/\sqrt{d}\leq 2/\sqrt{d}$.

In the remaining case $D_j$ is of degree $\delta_j\ge 2$, and then we can use the classical bounds for counting integral points of height at most $B$ on $D_j$ to end the case $n=2$, as follows.
Let $J\subset \NN$ consist of all $j$ for which $2\leq \delta_j\leq \log B$.
If $j\in J$ then by \Cref{thm:BCK}~\ref{item:1} there are at most $B^{1/2}\log (B)^{O(1)}$ integral points of height at most $B$ on $D_j$.
Also, since $F=g=0$ has degree at most $O_e(d^5B^{1/\sqrt{d}}\log B)$, we have that $\#J \ll_e d^5B^{1/\sqrt{d}}\log B$.
Hence there are at most
\[
\ll_e d^5B^{1/2 + 1/\sqrt{d}}\log (B)^{O(1)}
\]
integral points of height at most $B$ on all $D_j$ with $j\in J$, and hence similarly on the $C_j$ with $j\in J$.

Next, let $J'\subset \NN$ consist of all $j$ for which $\delta_j > \log B$.
For such a $j\in J'$, \Cref{thm:BCK}~\ref{item:1} again implies that there are at most
\[
\ll \delta_j^2 \log(B)^{O(1)}
\]
integral points of height at most $B$ on $D_j$.
For all of the $D_j$ with $j\in J'$ together, we then obtain an upper bound of the form
\begin{align*}
\sum_{j\in J'} \delta_j^2 \log(B)^{O(1)} &\ll \left(\sum_{j\in J'}\delta_j\right)^2 \log(B)^2 \ll_e d^{10}B^{2/\sqrt{d}}(\log B)^{O(1)},
\end{align*}
where we have used that $\sum_{j\in J'} \delta_j \leq \deg(F=g=0) = O_e(d^5B^{1/\sqrt{d}}\log B)$.
This concludes the proof for $n=2$.

Now assume that $n\geq 3$.
We will use Bertini's theorem to find a good hyperplane to intersect with, and count integral points on the intersections when moving this hyperplane.
Let $X \subset \PP(e, 1, ..., 1)$ be the subvariety cut out by $F_{\rm top}(Y, X)=0$.
By assumption, $X$ is geometrically irreducible.
Consider, on $\PP(e, 1, ..., 1)$, the global sections $X_1, ..., X_n$ of the sheaf $\mathcal{O}(1)$.
These can be pulled back to sections of a sheaf $\mathcal{L}=i^* \mathcal{O}(1)$ on $X$.
Now, for $v=a_1 X_1+...+a_n X_n$, denote by $X_v$ the intersection of $v=0$ with $X$ in $\PP(e, 1, ..., 1)$.
Note that because of the form of $F_{\rm top}$, we find that $\cap_{v} X_v =\emptyset$, which is definitely of codimension at least $2$ in $X$.

The projection $\pi: \AA^{n+1}\to \AA^n: (y,x)\mapsto x$ is finite-to-one and dominant when restricted to $X$, since $F_{\rm top}$ is monic in $Y$.
Therefore, $Z=X_{X_1}\cap X_{X_2}$ is of codimension $2$ in $X$, and $Z$ is not contained in $X_{X_3}$.
By \cite[Tag.\,0G4F]{stacks-project}, we find that for a Zariski-dense set of $v$, the variety $X_v$ is absolutely irreducible.

Next, choose such a $v$ for which $a_n =1$.
Substitute $X_n=a_1X_1+...+a_{n-1}X_{n-1}$ in the equation for $F_{\rm top}$.
We obtain a homogeneous polynomial $G(Y, X_1, \ldots, X_{n-1})$ of degree $de$ whose coefficients are polynomials in the $a_i$ of degree at most $de$.
Moreover, by our reasoning above $G$ is absolutely irreducible for a generic choice of $v$.
By~\cite[Lem.\,3.2.2]{CCDN-dgc}, there is a polynomial $\Phi(a_i)$ in the $a_i$ of degree at most $(de)^3$, such that if $\Phi(a_i)\neq 0$, then $G$ is absolutely irreducible.
By the Schwartz--Zippel bound from \Cref{prop:schwarz-zippel} there is some tuple of integers $a_i$ of height at most $(de)^3$ for which $G$ will be absolutely irreducible.
Then let $H$ be the hyperplane $X_n=a_1 X_1+...+a_{n-1} X_{n-1}$.

Finally, intersect $F=0$ with $k+H$ for $k$ varying from $-n(de)^3B$ to $n(de)^3B$.
For each such $k$, we get a new polynomial $F_k$ of the correct form for the induction hypothesis to apply.
We conclude that
\[
N_{\rm aff}(F; B^e, B, ..., B) \le \sum_{k=-n(de)^3B}^{n(de)^3B} N_{\rm aff}(F_k; B^e, B, ..., B) \ll_{n,e} d^{e+O(n)}B^{n-1},
\]
for $d\geq 5$.
For $d=2,3,4$ one argues in a completely identical manner to obtain the desired bound.
\end{proof}

We can now derive \Cref{thm:main-Serre} from \Cref{thm:mainFYX}, using in particular the reduction to specific covers provided by \Cref{lem:Broberg}.

\begin{proof}[Proof of \Cref{thm:main-Serre}]
Using Zariski's main theorem in the form of Grothendieck, we may assume that $\pi$ is finite.
By \Cref{lem:Broberg}, we can find a proper closed subset $V$ of $X$ and an irreducible $F\in \ZZ[Y, X_1, ..., X_n]$ for which $F(Y^e, X)$ is homogeneous of degree $de$ and which is monic in $Y$ for some $e\ge 1$ such that
\[
N_{\rm cover}(f, B) \le N(f(V), B)+N_{\rm cover}^{\rm aff}(F, B).
\]
For the first term we may simply apply \Cref{prop:schwarz-zippel} to $f(V)\subset \PP^n$, which has dimension at most $n-1$ and hence $N(f(V),B)\ll_f B^n$.
For the second term, we may assume by \Cref{lem:irre.abs.irre} that $F$ is absolutely irreducible, and then the result follows from \Cref{thm:mainFYX}.
\end{proof}

\begin{remark}
Note that the bounds of \Cref{thm:mainFYX} go wrong in more general affine situations than those of the theorem, see \cite[p.~177]{Serre-Mordell}. For example, look at $f(y,x) = y^2 + x_1 + x_2$ as a degree $2$ cover of $\AA^2$; note that this $f$ does not satisfy the assumptions of \Cref{thm:mainFYX}: its top degree weighted homogeneous part is not absolutely irreducible. In this example, counting on twisted lines goes wrong, as for each value of $y$ (of which there are $\sim B$) there is a line of points, each of which has $\sim B^2$ points.
\end{remark}

\section{Serre's question over global fields of any characteristic}\label{sec:global}

\subsection{Definitions and main results}

We use the common set-up from \cite{Pared-Sas}, \cite{CDHNV-dgc} and \cite{Verm:aff} to get a variant of \Cref{thm:mainFYX,thm:main-Serre,thm:mainFYX:unif} over any global field $K$ instead of just $\QQ$.
One of the differences is that we replace the above use of 
\cite{CCDN-dgc} by  \cite{CDHNV-dgc,Pared-Sas}. Another difference is that the constants from above may additionally depend on $K$. Let us recall this set-up here.

Let $K$ be a global field, by which is meant a finite separable extension of $\QQ$ or of $\FF_q(t)$ for some prime power $q$. If $K$ is an extension of $\FF_q(t)$ then we assume that $\FF_q$ is the full field of constants of $K$. Denote $k = \QQ$ or $\FF_q(t)$ and $d_K = [K : k]$.

We denote by $M_K$ the set of places of $K$.
If $K$ is a number field, denote by $M_{K, {\rm fin}}$ the set of finite places of $K$, corresponding to prime ideals of $\cO_K$, and by $M_{K, \infty}$ the infinite places of $K$, corresponding to embeddings $K\to \CC$.
In more detail, a finite place corresponds to a non-zero prime ideal $p$ of $\cO_K$, and the corresponding place $v$ is normalized by
\[
|x|_v = (\# \cO_K/p)^{-\ord_p(x) / d_K}.
\]
An infinite place $v$ comes from an embedding $\sigma: K\hookrightarrow \CC$, and is normalized by saying that
\[
|x|_v = |\sigma(x)|^{n_v/d_K},
\]
where $n_v = 1$ if $\sigma(K)\subset \RR$ and $n_v = 2$ otherwise.
If $K$ is a function field, separable over $\FF_q(t)$, then every place $v$ corresponds to a discrete valuation ring $\cO$ of $K$ containing $\FF_q$ with fraction field $K$.
If $p$ is the maximal ideal of $\cO$, then we normalize $v$ so that
\[
|x|_v = (\# \cO_K / p)^{-\ord_p(x) / d_K}.
\]
We fix a place $v_\infty$ above the place of $\FF_q(t)$ defined by $|f|_\infty = q^{\deg f}$, and define $M_{K, \infty} = \{v_\infty\}$ and $M_{K, {\rm fin}} = M_K\setminus \{v_\infty\}$.
The ring $\cO_K$ of integers of $K$ is the set of $x \in K$ for which $|x|_v \le 1$ for all places $v \not= v_\infty$.
If $x=(x_0 : \ldots : x_n)\in \PP^n(K)$ then we define the \emph{height of $x$} as usual to be
\[
H(x) = \prod_{v\in M_K} \max_i |x_i|_v,
\]
which is well-defined in view of the product formula.
The \emph{relative height} is defined by $H_K = H^{d_K}$.

If $X\subset \PP^n_K$ is a projective variety, then we define $N(X,B)$ to be the number of $K$-rational points $x$ on $X$ for which $H(x)\leq B$.
For the affine case, we define $[B]_{\cO_K}\subset \cO_K$
to consist of all $x\in \cO_K$ for which
\[
\begin{cases}
\max_{\sigma: K\hookrightarrow \CC} |\sigma(x)|\leq B^{1/d_K}, & \text{ if $K$ is a number field, or} \\
|x|_{v_\infty} \leq B^{1/d_K}, & \text{ if $K$ is a function field.}
\end{cases}
\]
If $X\subset \AA^n_K$ is an affine variety then we define $N_{\rm aff}(X,B)$ to be the cardinality of the set $X(K)\cap [B]_{\cO_K}^n$.

With this notation in place, we have the following general result for counting rational points on thin sets of type II. Note that when $K$ has positive characteristic in \Cref{thm:main.global}, we assume $f: X\to \PP_K^n$ to be separable, meaning that the corresponding extension $K(X)/K(\PP_K^n)$ of function fields is separable. The separability condition is used to adapt the proof of Serre's and Broberg's reduction recalled above as \Cref{lem:Broberg}. We refer to \cite[Section 8]{DaaDitFehm} for a discussion and further references on this separability condition related to thin sets. 

\begin{thm}\label{thm:main.global}
Let $K$ be a global field and let $f: X\to \PP^n$ be a quasi-finite, separable, dominant morphism of degree $d\geq 2$ defined over $K$, and assume that $X$ is integral.
Then
\[
\#\{x\in f(X(K))\mid H(x)\leq B\} \ll_{K,f} \begin{cases}
B^{nd_K} & \text{ if } d\geq 5, \\
B^{nd_K} \log(B)^{O(1)} & \text{ if } d = 4, \\
B^{(n-1+2/\sqrt{d})d_K}  \log(B)^{O(1)} & \text{ if } d=2,3.
\end{cases}
\]
\end{thm}

The main technical result is again a uniform bound for counting on weighted affine varieties, similar to \Cref{thm:mainFYX:unif}, as follows.

\begin{thm}\label{thm:main.unif.global}
Let $K$ be a global field and let $n\geq 2, d\geq 2, e\geq 1$ be integers.
Let $F_{\rm top}\in K[Y, X_1, \ldots, X_n]$ be weighted homogeneous of degree $de$ with weights $(e, 1, \ldots, 1)$ and monic in $Y$, let $F_0\in K[Y, X_1, \ldots, X_n]$ be such that $F_0(Y^e,X)$ has degree strictly less than $de$, and define $F = F_{\rm top} + F_0$. Suppose that $F_{\rm top}$ is absolutely irreducible.
Then
\[
N_{\rm aff}(F; B^e, B, \ldots, B)\ll_{K,n,e} \begin{cases}
d^{e+O(n)}B^{n-1} & \text{ if } d\geq 5, \\
B^{n-1} \log(B)^{O(1)} & \text{ if } d = 4, \\
B^{n-2+2/\sqrt{d}}  \log(B)^{O(1)} & \text{ if } d=2,3,
\end{cases}
\]
where $N_{\rm aff}(F; B^e, B, \ldots, B)$ is the number of $(y, x_1, \ldots, x_n)\in [B^e]_{\cO_K}\times [B]_{\cO_K}^n$ for which $F(y, x_1, \ldots, x_n) = 0$.
\end{thm}

We leave the adaptation of \Cref{thm:mainFYX} to all global fields to the reader, as it follows directly from \Cref{thm:main.unif.global}.

\subsection{The case of curves}
In the case of a thin subset of type II of the projective line one can do better than \Cref{thm:main.global}, similar to the results and proofs of Section 9.7 in \cite{Serre-Mordell} which we slightly generalize here to any global field and which we include for didactical reasons.
We use the following basic fact about heights on curves over global fields and we assume knowledge of Weil's height machine, see e.g.~\cite[Chapter 2, in particular \S2.8]{Serre-Mordell}.
	
\begin{lem}\label[lem]{lem:height_curves}
	Let $X/K$ be a curve and $D, D'$ non-zero effective divisors on  $X$. Let $H_D,  H_{D'}$ be height functions associated respectively to $D, D'$. For all $\varepsilon > 0$ and all $x \in X(K)$ we have
	\[
	H_D(x) \ll_{X,\varepsilon} H_{D'}^{\deg D/\deg D' + \varepsilon}(x).
	\]
	Moreover, if $X$ has genus $0$ then this even true for $\varepsilon = 0$.
\end{lem}
\begin{proof}	
	If $C$ has genus $0$ then the divisor class of $D_0 := \deg D' D - \deg D D'$ is trivial so the corresponding height $H_0 := H_D^{\deg D'}/ H_{D'}^{\deg D}$ is a bounded function and the lemma follows.
	
	In all other cases we consider the $\RR$-divisor $D_{\varepsilon} := (\deg D/\deg D' + \varepsilon) D' - D$. This divisor has degree $\varepsilon > 0$, so is linearly equivalent to an effective divisor which implies that the corresponding height function $H_{D'}^{\deg D/\deg D' + \varepsilon}/H_D$ is bounded from below.
\end{proof}

\begin{prop}\label[prop]{thm:PP1}
	Let $X$ be integral and $f: X \to \PP^1_K$ be a quasi-finite cover of degree $d \ge 1$. Then $\# \{x \in f(X(K)) \mid H(x) \leq B\} = O_f(B^{\frac{2d_K}{d}})$.
\end{prop}
\begin{proof}
	We may assume that $f$ is finite by Zariski's main theorem. This implies that $X$ is normal. If $X$ is not geometrically integral then $X(K) = \emptyset$. Similarly, if $K$ is a function field then $X$ could be a normal but non-smooth curve. In this case $X(K)$ is finite by \cite[Theorem 3]{Voloch1991}. We may thus assume that $X$ is a smooth curve and do a case analysis based on the genus of $X$.
	
	If $X$ has genus $> 1$ then let $H_X$ be the exponential canonical height considered in \cite[\S5.6]{Serre-Mordell}, i.e~coming from a theta divisor. It then follows from Mumford's theorem \cite[\S5.7, second Corollary]{Serre-Mordell} that
	\[
	\#\{x \in X(K): H_X(x) \leq B\} \ll_X (\log \log B).
	\]
	It follows from Lemma \ref{lem:height_curves} applied to $H_X$ and the pullback height $f^*H$ that there exists $C,D > 0$ such that $f^* H(x) \leq C H_X(x)^D$ for all $x \in X$ which is sufficient.
	
	If $X$ has genus $1$ then we may assume that $X$ is an elliptic curve since $X(K) = \emptyset$ otherwise. The Mordell-Weil theorem implies that $X(K)$ is a finitely generated abelian group. Let $H_X$ be the canonical exponential height \cite[\S VIII.9]{Silverman2009} on $X$, its logarithm is a quadratic form on the Mordell-Weil lattice $X(K)$ \cite[Proposition 9.6]{Silverman2009}. It is well-known that this formally implies that
	\[
	\#\{x \in X(K): H_X(x) \leq B\} \ll_X (\log B)^C.
	\]
	for some $C > 0$. We can then again use Lemma \ref{lem:height_curves} to compare the heights $H_X$ and $f^* H$.
		
	If $X$ has genus $0$ then we may assume that $X(K) \neq \emptyset$ and thus that $X = \PP^1$. Let $H_{X}$ be the standard height. The pullback height $f^*H$ corresponds to a divisor of degree $d$ so there exists $C > 0$ such that $f^*H(x) \leq C H_{X}(x)^d$ for all $x \in \PP^1(K)$. It is well-known that \footnote{Precise asymptotic formulas are known, see \cite{Schanuel1979Heights} for number fields and \cite{Dipippo1990Spaces} for global function fields.}
	\[
	\# \{x \in X(K): H_{X}(x) \leq B\} \ll_K B^{2d_K}.
	\]
	The proposition follows after substituting $B$ by $(B/C)^{1/d}$.
	\end{proof}

\subsection{Proofs over global fields}

Let us explain how to adapt our proofs to a general global field $K$.

Firstly, all results from \Cref{sec:prelims} hold in this general context as well.
For \Cref{prop:schwarz-zippel} this is proved in exactly the same way.
\Cref{thm:BCK} is in fact already stated for global fields in~\cite[Thms.\,4, 5, Cor.~12]{BinCluKat}.
For number fields the proof of \Cref{lem:irre.abs.irre} is exactly the same, while for function fields one can follow the argument of~\cite[Cor.\,1]{Heath-Brown-Ann}.
The Bombieri--Vaaler theorem is stated more generally for number fields in~\cite{Bombi-Vaal}, while for function fields this follows e.g.\ from~\cite{Hess}.
Finally, the proof of \Cref{lem:Broberg} given in~\cite[Lem.\,3]{Broberg03} works identically when $K$ is a global field and $f$ is separable.

For the existence of auxiliary polynomials we have the following.
For $f$ in $K[X_1, \ldots, X_n]$ we can write $f = \sum_i a_i X^i$ and we define
\[
H_K(f) = \prod_{v\in M_K}\left(\max_i |a_i|_v\right)^{d_K}
\]
to be the \emph{height of $f$}.
Let $\Sigma\subset \PP(w)(K)$ consist of all points $P$ such that for some non-zero prime $p$ of $\cO_K$ we have that $P$ is singular on the reduction of $\PP(w)$ modulo $p$.

\begin{thm}\label{thm:aux.global}
Let $n\geq 2$, write $m=n-1$, and let $w = (w_0, \ldots, w_n)$ be a weight vector.
Let $c = 4$ if $\charac K = 0$ and $c=10$ if $\charac K > 0$.
Assume that $w_n = 1$, let $f\in \cO_K[X_0, \ldots, X_n]$ be full, absolutely irreducible, and weighted homogeneous of degree $d$ with weights $w$, and let $B\geq 2$.
Then there exists a weighted homogeneous polynomial $g\in \cO_K[X_0, \ldots, X_n]$ coprime to $f$, vanishing on
\[
\{(x_0 : \ldots : x_n)\in V(f)(\QQ)\mid H(x_i)\leq B^{w_i}\} \setminus \Sigma,
\]
and of weighted degree
\[
\ll_{K,w} d^{c-1/m}B^{\frac{m+1}{m}\frac{|w|^{1/m}}{d^{1/m}}}\frac{b(f)}{H_K(f)^{\frac{|w|^{1/m}}{md^{1+1/m}}}} + d^2\log B  + d^{c-\frac{1}{m}} .
\]
\end{thm}

With a more careful analysis one can likely improve the exponent of $d$ in the above result, similar to~\cite{Pared-Sas}.

\begin{proof}
The proof is essentially identical to the proof of \Cref{thm:aux.general}, but using the generalities from~\cite{Pared-Sas} over global fields.
Let us give some more details.

For the estimates on the determinant from \ref{prop:det.est} one should follow the approach from~\cite[Sec.\,3]{Pared-Sas}.
In particular generalizing~\cite[Thm.\,3.16]{Pared-Sas} to weighted homogeneous polynomials.
Via explicit Noether polynomials, this is where the different values of $c$ arise in characteristic zero and positive characteristic.
The required change of coordinates follows from~\cite[Sec.\,4]{Pared-Sas}, which allows one to assume that $f$ has a coefficient which is large in height compared to the height of $f$.
Note that this uses the assumption that $w_n=1$.
Finally, the construction of the auxiliary polynomial is then similar to the proof of \Cref{thm:aux.general}, see also~\cite[Thm.\,5.2]{Pared-Sas}.
\end{proof}

In the affine case we have the following.

\begin{thm}\label{thm:aux.aff.global}
Define $c = 2$ if $\charac K = 0$ and $c = 8$ if $\charac K > 0$.
Let $f\in \cO_K[X_1, \ldots, X_n]$ be absolutely irreducible, full, of weighted degree $d$ with weights $w$, and let $B\geq 2$ be a positive integer.
Define $w' = (w,1)$ and define $\Sigma\subset \PP(w)(K')$ as above.
Write $m = n-1$.
Then there exists a polynomial $g\in \cO_K[X_1, \ldots, X_n]$ coprime to $f$, vanishing on
\[
\left(\left([B^{w_1}]_{\cO_K}\times \ldots [B^{w_n}]_{\cO_K}\right)\cap V(f)\right)\setminus \Sigma,
\]
and of degree
\[
\ll_{K,w} d^{c-1/m} B^{\frac{|w|^{1/m}}{d^{1/m}}} \frac{\min \{ \log H_K(f_d) + d\log B + d^2, d^2b(f) \}}{H_K(f_d)^{\frac{1}{md^{1+1/m}}}} + d^{2+c-1/m} \log B.
\]
\end{thm}

\begin{proof}
This again follows along similar lines as the proof of \Cref{thm:aux.affine.general}, but using the techniques from~\cite[Sec.\,5B]{Pared-Sas}.
\end{proof}

We can handle the twisted lines in a similar way as in \Cref{prop:tw-l}, note that the definition of a twisted line works over any field.
\begin{prop}\label[prop]{prop:tw-l-global}
	 Let $F = F_{\mathrm{top}} + F_0 \in \mathcal{O}_K[Y, X_1, X_2]$ be as in \Cref{sec:twisted}, but over $\cO_K$ instead of over $\ZZ$. Namely, $F_{\mathrm{top}}$ is absolutely irreducible, weighted homogeneous of degree $de$ with weights $(e, 1, 1)$ and monic in $Y$, and $F_0(Y^e, X)$ has degree $< de$.
	
	 Let $I$ be a finite collection of twisted lines on $X = V(F)\subset \AA^3$ over $K$.
	Then
	\[
	N_{\mathrm{aff}}(\cup_{\ell\in I} \ell, B^e, B, B) \ll_{K} (ed)^{e + 8} B + e \# I
	\]
	if $d>2$. If $d=2$ then one has
	\[
	N_{\mathrm{aff}}(\cup_{\ell\in I} \ell, B^e, B, B) \ll_K   (ed)^{e + 8}B \log(B) + e \# I.
	\]
\end{prop}
\begin{proof}
	If $K$ is a number field then fix a place $v_{\infty} \in M_{K, \infty}$. To prevent confusion recall that $[B]_{\mathcal{O}_K}$ uses $|\sigma(x)|$ for a complex embedding $\sigma: K \to \CC$ in  the number field case. This is equal to $|x|_v^{d_K/n_v}$ for the corresponding archimedean place $v$.
	
	\Cref{lem:tw-l} generalizes except for the fact that we may not assume that $\gcd(v_1, v_2)$ equals $1$ as the class group of $\mathcal{O}_K$ can be non-trivial. Instead we use that the class group is finite, fix integral ideals which represent each ideal class and we may assume that $v_1, v_2$ generate one of these representatives. By multiplying $v_1, v_2$ by a unit an extra condition which we may assume is
	\begin{equation}\label{eq:infinite_places_equal_role}
		1 \ll_K \max(|v_1|_v, |v_2|_v) \ll_K 1
	\end{equation}
	for all $v \neq v_{\infty} \in M_{K, \infty}$. This follows from Dirichlet's unit theorem.
	
	For the proof of \Cref{lem:tw-count} note that we may not assume that $t$ is integral, only that it is contained in the inverse of one of the chosen representatives of ideal classes. The rest of the proof generalizes immediately over function fields except that we use the algebraic closure of the completion of $K$ at $v_{\infty}$ instead of $\CC$. For number fields $K$ we only consider the complex embedding defined by $v_{\infty}$, the argument shows that $|t - \alpha|_{v_{\infty}} \leq B^{\frac{1}{d_k}}/\max\{|w_e|_{v_{\infty}}^{\frac{1}{e}}, |v_{1}|_{v_{\infty}}, |v_2|_{v_{\infty}}\}$ for one of $e + 1$ choices of $\alpha \in \CC$. On the other hand for $v \neq v_{\infty} \in M_{K, \infty}$ we have $|a_1 - v_1 t|_v, |a_2 - v_2 t|_v \leq B^{\frac{1}{d_K}}$ and \eqref{eq:infinite_places_equal_role}. Theorem 2.6.7 of \cite{Lang-dioph} then gives us the upper bound
	\[
	N_{\rm aff} (\ell; B^e, B, B) \ll_K  \frac{eB}{ \max\{|w_e|_{v_{\infty}}^{{d_K}/e}, |v_1|_{v_\infty}^{d_K}, |v_2|_{v_{\infty}}^{d_K}\}}  + e.
	\]

	The proof of \Cref{lem:tw-l:few} is completely geometric and works over any field. The analogue of \Cref{lem:tw-l:rat} is proven by replacing \Cref{thm:aux.general} with \Cref{thm:aux.global} and \Cref{lem:b(f).bound} by Lemma 3.23 of \cite{Pared-Sas}. Note that we have an extra factor of $(ed)^{4}$ by this application.
	
	We conclude by describing how one should adapt the proof of \Cref{prop:tw-l}. The relevant claim is that if $\ell \cap ([B^e]_{\cO_K} \times [B]_{\cO_K}^2)$ has at least $e + 1$ points then $\max\{|w_{e}|_{v_{\infty}}^{1/e}, |v_1|_{v_{\infty}}, |v_2|_{v_{\infty}} \} \ll_K B^{\frac{1}{d_K}}$.
	
	Let $(y_1, x_{1,1}, x_{2,1}), \dots (y_{e + 1}, x_{1,e + 1}, x_{2,e + 1})$ be the points and $t_1, \dots, t_{e + 1}$ the corresponding $t$. To apply the Lagrange interpolation argument we need to bound $|t_{i} - t_j|_{v_{\infty}}$ from below for all $i \neq j$. Assume without loss of generality that $v_1 \neq 0$. For all places $v \neq v_{\infty}$ we have because $v_1 \in \mathcal{O}_K$ and \eqref{eq:infinite_places_equal_role} that
	\[
	|t_{i} - t_j|_{v} = |v_1|_v^{-1} |x_{1, i} - x_{1, j}|_v \ll_K \begin{cases}
		1 \text{ if $v$ finite} \\
		B^{\frac{1}{d_K^{2}}} \text{ if $v$ infinite}.
	\end{cases}
	\]
	
	The product formula then implies that $|t_{i} - t_j|_{v_{\infty}} \gg_K 1$, resp.~$|t_{i} - t_j|_{v_{\infty}} \gg_K B^{\frac{1 - d_K}{d_K^{2}}}$, if $K$ is a function field, respectively a number field. Putting this into the Lagrange interpolation formula shows that $|w_e|_{v_{\infty}} \ll_K e B^{\frac{1}{d_K}}$. We are done as $e^{1/e} = O(1)$. A similar but simpler argument shows that $|v_1|_{v_{\infty}}, |v_2|_{v_{\infty}} \ll_K B^{\frac{1}{d_K}}$.
	
	The summmation by parts argument is completely identical and leads to the upper bound in the statement.
\end{proof}

With these auxiliary results in place, the proofs of \Cref{thm:main.global} and \Cref{thm:main.unif.global} are now completely similar to the proofs over $\QQ$.

\begin{proof}[Proof of \Cref{thm:main.unif.global} and \Cref{thm:main.global}]
The proof is again by induction on $n$, where the key case is $n=2$.
For this, \Cref{thm:aux.aff.global} provides us with an auxiliary polynomial $g$ catching all points of $\{(x_0 : \ldots : x_n)\in V(f)(\QQ)\mid H(x_i)\leq B^{w_i}\}$ of controlled degree.
We count on the irreducible components of $f=g=0$, which are all of dimension one, in a completely identical way to the proof of \Cref{thm:mainFYX:unif}, using \Cref{prop:tw-l-global} for the twisted lines instead of \Cref{prop:tw-l}.
The induction argument uses Bertini's theorem in a similar way to allow one to slice with hyperplanes.

\Cref{thm:main.global} now follows easily from \Cref{thm:main.unif.global} using Broberg's lemma (see \Cref{lem:Broberg}), whose proof adapts easily to any global field.
\end{proof}

\bibliographystyle{amsalpha}
\bibliography{anbib}

\end{document}